\documentclass{amsart}

\usepackage{amsmath,amssymb,amsthm,amscd}

\usepackage{color}

\usepackage[nohug,heads=LaTeX]{diagrams}

\usepackage{pinlabel}

\newtheorem{prop}{Proposition}[section]
\newtheorem{thm}[prop]{Theorem}
\newtheorem{lem}[prop]{Lemma}
\newtheorem{cor}[prop]{Corollary}

\theoremstyle{definition}

\newtheorem*{defn}{Definition}
\newtheorem{ex}[prop]{Example}
\newtheorem{rem}[prop]{Remark}

\newtheorem*{ack}{Acknowledgements}

\def\co{\colon\thinspace}

\newcommand{\C}{\mathbb{C}}
\newcommand{\CP}{\mathbb{C}\mathrm{P}}

\newcommand{\rmd}{\mathrm{d}}

\newcommand{\rme}{\mathrm{e}}

\newcommand{\End}{\mathrm{End}}

\newcommand{\F}{\mathbb{F}}

\newcommand{\Ham}{\mathrm{Ham}}

\newcommand{\rmi}{\mathrm{i}}

\newcommand{\MM}{\mathcal{M}}
\newcommand{\wtM}{\widetilde{M}}

\newcommand{\N}{\mathbb{N}}
\newcommand{\NN}{\mathcal{N}}

\newcommand{\Open}{\mathrm{Open}}

\newcommand{\bfp}{\mathbf{p}}
\newcommand{\mfp}{\mathfrak{p}}

\newcommand{\bfq}{\mathbf{q}}

\newcommand{\R}{\mathbb{R}}

\newcommand{\RP}{\mathbb{R}\mathrm{P}}

\newcommand{\wtW}{\widetilde{W}}

\newcommand{\xist}{\xi_{\mathrm{st}}}
\newcommand{\wtX}{\widetilde{X}}

\newcommand{\bfz}{\mathbf{z}}

\newcommand{\Z}{\mathbb{Z}}

\DeclareMathOperator{\Diff}{Diff}

\DeclareMathOperator{\ev}{ev}

\DeclareMathOperator{\Int}{Int}

\DeclareMathOperator{\Symp}{Symp}

\DeclareMathOperator{\Wh}{Wh}

\begin{document}

\author[K.~Barth]{Kilian Barth}
\address{Mathematisches Institut, WWU M\"unster,
Einstein\-stra\-{\ss}e 62, 48149 M\"unster, Germany}
\email{k{\_}bart05@uni-muenster.de}

\author[H.~Geiges]{Hansj\"org Geiges}
\address{Mathematisches Institut, Universit\"at zu K\"oln,
Weyertal 86--90, 50931 K\"oln, Germany}
\email{geiges@math.uni-koeln.de}

\author[K.~Zehmisch]{Kai Zehmisch}
\address{Mathematisches Institut, WWU M\"unster,
Einstein\-stra\-{\ss}e 62, 48149 M\"unster, Germany}
\email{kai.zehmisch@uni-muenster.de}

\title{The diffeomorphism type of symplectic fillings}

\date{}

\begin{abstract}
We show that simply connected contact manifolds that are subcritically
Stein fillable have a unique symplectically aspherical
filling up to diffeomorphism. Various extensions to
manifolds with non-trivial fundamental group are discussed.
The proof rests on homological restrictions on symplectic
fillings derived from a degree-theoretic analysis of the evaluation
map on a suitable moduli space of holomorphic spheres. Applications
of this homological result include a proof that compositions of right-handed
Dehn twists on Liouville domains are of infinite order in the
symplectomorphism group. We also derive uniqueness results
for subcritical Stein fillings up to homotopy equivalence and,
under some topological assumptions on the contact manifold,
up to diffeomorphism or symplectomorphism.
\end{abstract}


\thanks{H.~G.\ and K.~Z.\ are partially supported by DFG grants
GE 1245/2-1 and ZE 992/1-1, respectively.}

\maketitle


\section{Introduction}
The aim of this paper is to study the topology of symplectic fillings
$(W,\omega)$ of a given contact manifold $(M,\xi)$. By
`symplectic filling' we always mean a \emph{strong}
filling~\cite[Definition~5.1.1]{geig08}.

Any filling can of course be modified by performing
symplectic blow-ups; this is ruled out if we require the
filling $(W,\omega)$ to be \emph{symplectically aspherical}, that is,
$[\omega]|_{\pi_2(W)}=0$.

It is well known that even under this asphericity assumption
symplectic fillings are not, in general,
unique up to diffeomorphism. For instance, McDuff~\cite{mcdu90}
observed that the lens space
$L(4,1)$ with its standard contact structure coming from
the $3$-sphere
is Stein fillable both by the disc bundle
over the $2$-sphere $S^2$ with Euler class~$-4$, and by the
complement of the quadric in $\CP^2$, cf.~\cite[Exercises 12.3.4]{ozst04}.
Many more examples of lens spaces with non-unique fillings
have been found by Lisca~\cite{lisc04}.

On the other hand, there are also contact manifolds whose
symplectic fillings are determined up to diffeomorphism, or even
symplectomorphism.
The first result about the diffeomorphism type of fillings
(in all dimensions) is due to Eliashberg--Floer--McDuff,
see \cite[Theorem~1.5]{mcdu91}.
Here $\xist$ denotes the standard contact structure on
the odd-dimensional standard sphere coming
from the obvious filling by the standard symplectic ball.

\begin{thm}[Eliashberg--Floer--McDuff]
\label{thm:EFM}
Let $(W,\omega)$ be a symplectically aspherical filling of
$(S^{2n-1},\xist)$, $n\geq 3$. Then $W$ is diffeomorphic to the ball~$D^{2n}$.
\end{thm}

Earlier, it had been proved by Gromov~\cite[p.~311]{grom85}
and McDuff~\cite[Theorem~1.7]{mcdu90}, using positivity
of intersection in dimension four,
that any symplectically aspherical
filling of $(S^3,\xist)$ is even symplectomorphic to the 
standard $4$-ball.

In this paper, we study the topology of symplectically
aspherical fillings of contact manifolds that
admit a subcritical Stein filling, that is, where the
plurisubharmonic function on the Stein filling has handles 
of index below the middle dimension only.
Essentially, what we show (under various topological
assumptions) is that the existence of a single subcritical
Stein filling fixes the diffeomorphism type of all
symplectically aspherical fillings.

By analysing the moduli space of holomorphic spheres in
a partial compactification of the filling, we derive a degree-theoretic
statement concerning the evaluation map on this moduli space.
This approach was pioneered by McDuff in \cite{mcdu91},
and developed further by two of the present authors in a number
of papers, e.g.\ \cite{geze12,geze16a,geze16b}.
From the evaluation map on the moduli space
we derive a homological vanishing result for fillings, which then
leads to the following result. Here, by slight abuse of notation,
we write the Stein filling as a pair $(W_0,\omega_0)$ consisting
of a manifold and a symplectic form, since we are primarily interested
in the symplectic properties of fillings.

\begin{thm}
\label{thm:homology}
Let $(M,\xi)$ be a $(2n-1)$-dimensional closed, connected
contact manifold, $n\geq 2$,
admitting a subcritical Stein filling $(W_0,\omega_0)$ with
the homotopy type of a $CW$ complex of dimension $\ell_0\leq n-1$. Let
$(W,\omega)$ be any symplectically aspherical filling of~$(M,\xi)$. Then
the following holds:
\begin{itemize}
\item[(a)]
The integral homology groups of $W$ are
\[ H_k(W)\cong\begin{cases}
H_k(M) & \text{for $k=0,\ldots,\ell_0$},\\
0      & \text{otherwise,}
\end{cases} \]
where the isomorphism between the relevant homology groups of $M$ and
$W$ is induced by the inclusion $M\subset W$.
In particular, the homology groups of $W$ coincide with those
of~$W_0$.

\item[(b)] The inclusion $M\subset W$ is $\pi_1$-surjective. If the
fundamental group $\pi_1(M)$ is abelian, then the inclusion $M\subset W$
is also $\pi_1$-injective.
\end{itemize}
\end{thm}

This theorem will be proved in Section~\ref{section:proof-homology}.
Various direct applications are discussed in
Section~\ref{section:applications}.

\begin{rem}
In the proof of Theorem~\ref{thm:homology}, we appeal to a result
of Cieliebak, which relies on
the assumption that the filling is subcritical as a Stein manifold,
i.e.\ there are no Stein handles of critical index. It is probably not
sufficient, in general, merely to assume that the
Stein filling is of subcritical homotopical dimension (that is,
where the critical handles cancel topologically). Stein manifolds
of this kind exist by the work of Seidel--Smith~\cite{sesm05}
and McLean~\cite{mcle09}: in all even dimensions $2n\geq 8$
there are infinitely many distinct finite type Stein manifolds
diffeomorphic to~$\R^{2n}$; they all have Stein handle
decompositions involving critical handles.
\end{rem}

Our most significant application of Theorem~\ref{thm:homology}
is the following vast extension of results of Seidel
on generalised Dehn twists. This is discussed
in Section~\ref{section:DS}, where all the relevant concepts
will be introduced. For other work in this direction see
Remark~\ref{rem:DS}.

\begin{thm}
\label{thm:DS}
A (non-empty) composition of right-handed Dehn twists on a Liouville
manifold of dimension at least four is never isotopic to the
identity within the group of compactly supported
symplectomorphisms.
\end{thm}

Theorem~\ref{thm:homology} is also one of the essential
steps towards the main result of this paper, proved in
Section~\ref{section:filling}, about the
topological classification of symplectic fillings.

\begin{thm}
\label{thm:filling}
Let $(M,\xi)$ be as in Theorem~\ref{thm:homology}, $n\geq 3$. If $M$
is simply connected, then all symplectically aspherical
fillings of $(M,\xi)$ are diffeomorphic.
\end{thm}

An extension of this result to certain finite fundamental groups
is given in Theorem~\ref{thm:filling2}.

In Section~\ref{section:coverings} we generalise the
argument used to prove Theorem~\ref{thm:homology} to
the setting of coverings. This is used in Section~\ref{section:handle}
to derive results on the homotopy and diffeomorphism type
of fillings when the maximal index of a handle decomposition
of $W$ is known. One result that is easy to state
is the following.

\begin{thm}
\label{thm:scS-homotopy}
All subcritical Stein fillings of a closed, connected
contact manifold are homotopy equivalent.
\end{thm}

Of course the statement may be empty if the given contact manifold
does not admit any subcritical Stein fillings.

In Section~\ref{section:simple} we consider fillings of simple
manifolds~$M$. Recall that a topological space is called
\emph{simple} if its fundamental group acts trivially on all
its homotopy groups. Examples of simple manifolds are Lie groups
or, more generally, any manifold that is an
$H$-space~\cite[Corollary~$8^{bis}.3$]{mccl01}.
The main result of Section~\ref{section:simple} is the following.

\begin{thm}
\label{thm:simple}
Let $(M,\xi)$ be as in Theorem~\ref{thm:homology}, $n\geq 3$.
If $M$ is a simple space whose fundamental group
has vanishing Whitehead group, then all symplectically
aspherical fillings of $M$ are diffeomorphic.
\end{thm}

In Section~\ref{section:cotangent} we apply our theory to
the sphere bundle of stabilised cotangent bundles.
Finally, in Section~\ref{section:symplecto} we show that
if $(M,\xi)$ admits a $2$-subcritical Stein filling,
then all flexible Stein fillings (e.g.\ subcritical ones)
are symplectomorphic.

For other results on the topology of Stein fillings see
\cite[Chapter~16]{ciel12} and \cite[Chapter~12]{ozst04}
and the references therein.
\section{Proof of Theorem~\ref{thm:homology}}
\label{section:proof-homology}
\subsection{A completion of the filling}
\label{subsection:completion}
According to a theorem of Cie\-lie\-bak \cite{ciel02},
\cite[Section~14.4]{ciel12}, every subcritical Stein manifold is
deformation equivalent (hence symplectomorphic) to a split one.
Thus, if $(W_0,\omega_0)$ is the given subcritical Stein filling of
$(M,\xi)$, we may assume --- perhaps after scaling $\omega_0$ by a small
positive constant --- that there is a $(2n-2)$-dimensional Stein manifold
$(V,J_V)$ with plurisubharmonic function $\psi_V$ (with
$\min\psi_V=0$) and symplectic form $\omega_V=-\rmd(\rmd \psi_V\circ J_V)$,
such that $W_0$ is a sublevel set of the plurisubharmonic potential
\[ \psi(v,z):=\psi_V(v)+\frac{1}{4}\log\bigl(1+|z|^2\bigr) \]
on $(V\times\C,J_V\oplus\rmi)$, and such that, with $z=r\rme^{\rmi\theta}$,
the corresponding symplectic form
\[ \omega:=\omega_V+\frac{r\,\rmd r\wedge\rmd\theta}{(1+r^2)^2} \]
on $V\times\C$ coincides with $\omega_0$ on $W_0$ under the inclusion
$W_0\subset V\times\C$.

Now, given a symplectically aspherical filling
$(W,\omega)$ of $(M,\xi)$, we define the symplectic manifold
\[ (Z,\Omega):=(W,\omega)\cup_{(M,\xi)}\Bigl((V\times\C)
\setminus\Int(W_0), \omega_V+
\frac{r\,\rmd r\wedge\rmd\theta}{(1+r^2)^2}\Bigr).\]
Our choice of plurisubharmonic potential on the $\C$-factor is explained
by the fact that we can now
build a new symplectic manifold $(\hat{Z},\hat{\Omega})$
from $(Z,\Omega)$ by compactifying $\C$ to a complex projective
line $\CP^1=\C\cup\{\infty\}$ with its standard Fubini--Study
symplectic form of total area~$\pi$, see Figure~\ref{figure:completion}.

\begin{figure}[h]
\labellist
\small\hair 2pt
\pinlabel $(W,\omega)$ at 92 70
\pinlabel $(M,\xi)$ [bl] at 418 262
\pinlabel $\CP^1$ [bl] at 423 330
\pinlabel $V\times\{\infty\}$ [bl] at 505 189
\endlabellist
\centering
\includegraphics[scale=0.55]{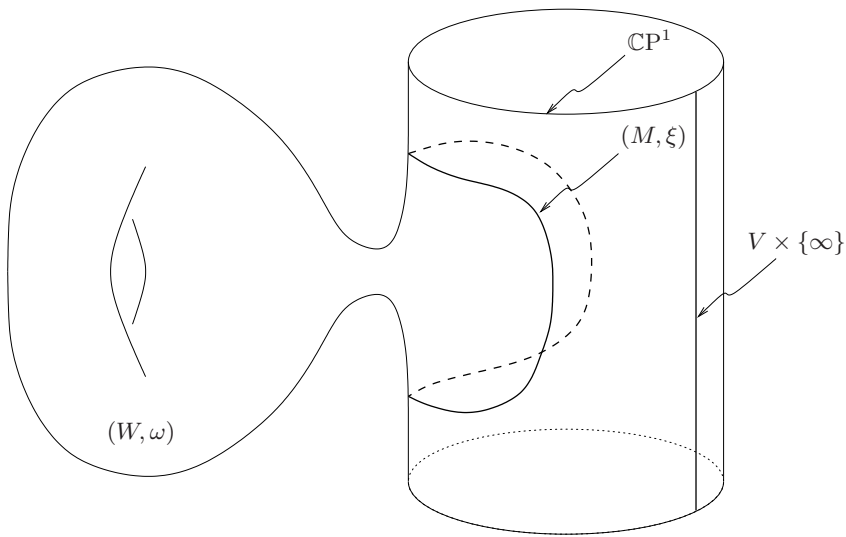}
  \caption{The symplectic manifold $\hat{Z}$.}
  \label{figure:completion}
\end{figure}

With $\omega_0$ scaled sufficiently small, we may assume that $M=\partial W_0$
is a level set of $\psi$ below $(\log 2)/4$, so that the hypersurfaces
$V\times\{\pm 1\}$ may also be regarded as subsets of~$\hat{Z}$.
Moreover, without loss of generality we may assume that $\psi_V$
has no critical points above the level of~$M$. This ensures that
the homotopical dimension of $V$ equals that of~$W_0$, i.e.\ $\ell_0$,
and $W$ is a strong deformation retract of~$Z$.

We equip the symplectic manifold $(\hat{Z},\hat{\Omega})$
with a compatible almost complex structure $J$, generic
in the sense of \cite{mcsa04} on
$\Int(W_0)$, and equal to $J_V\oplus\rmi$ on $\hat{Z}\setminus\Int(W)$.
\subsection{Holomorphic spheres}
For $v\in V$ with $\psi_V(v)>(\log 2)/4$ we have the obvious holomorphic spheres
$\{v\}\times\CP^1\subset (\hat{Z},J)$, which foliate the
corresponding part of~$\hat{Z}$.

\begin{lem}
\label{lem:spheres}
Let $u\co\CP^1\rightarrow\hat{Z}$ be a non-constant $J$-holomorphic sphere.
\begin{itemize}
\item[(i)] If $u(\CP^1)$ is contained in $\hat{Z}\setminus\Int(W)$,
then $u$ is a holomorphic branched covering $\CP^1\rightarrow
\{v\}\times\CP^1$ for some $v\in V$.
\item[(ii)] If $u(\CP^1)$ intersects $\hat{Z}\setminus\Int(W)$, then
it also intersects the hypersurface $H:=V\times\{\infty\}$.
\item[(iii)] If $u(\CP^1)$ intersects $\{\psi_V>(\log 2)/4\}\times\CP^1$,
then $u$ is as in~(i).
\end{itemize}
\end{lem}

\begin{proof}
(i) This follows from the maximum principle for $J_V$-holomorphic curves
in~$V$, applied to the $V$-component of~$u$.

(ii) The plurisubharmonic function $\psi$ is defined on
$\hat{Z}\setminus(\Int(W)\cup H)$ and a collar
neighbourhood of $M$ in~$W$, and the symplectic form
$\omega$ is exact in this region. Thus, if $u(\CP^1)$
intersects $\hat{Z}\setminus\Int(W)$, but not the hypersurface~$H$,
the maximum principle constrains $u(\CP^1)$ to lie
in a level set of~$\psi$, and hence inside the region where $\omega$
is exact. This forces $u$ to be constant.

(iii) Apply the maximum principle to the
$V$-component of $u$ on the preimage of $\{\psi_V>(\log 2)/4\}\times\CP^1$.
\end{proof}

Let $\MM$ be the moduli space
of holomorphic spheres $u\co\CP^1\rightarrow(\hat{Z},J)$ with the
properties
\begin{itemize}
\setlength\itemsep{.3em}
\item[(M1)] $[u]=\bigl[\{v\}\times\CP^1\bigr]$ in the homology group
$H_2(\hat{Z})$, where $v\in V$ can be any point with
$\psi_V(v)>(\log 2)/4$;
\item[(M2)] $u(z)\in V\times\{z\}$ for $z\in\{\pm1,\infty\}$.
\end{itemize}
Observe that a holomorphic sphere $u\in\MM$
satisfying one of the assumptions (i) or (iii) in Lemma~\ref{lem:spheres}
is simply an inclusion map $\CP^1\rightarrow\{v\}\times\CP^1$.

\begin{prop}
The moduli space $\MM$ is an oriented manifold of dimension~$2n-2$.
\end{prop}

\begin{proof}
This is proved exactly as \cite[Proposition~6.1]{geze12}. Notice that
our moduli space $\MM$ corresponds to $\MM_{-1,1,\infty}$
in \cite[p.~277]{geze12}.
\end{proof}

In the following proposition, the degree of a proper map between
non-compact oriented manifolds is understood in the sense
of \cite[Exercise 5.1.10]{hirs76}. We think of $\CP^1$
as $\C\cup\{\infty\}$.

\begin{prop}
\label{prop:ev}
The evaluation map
\[ \begin{array}{ccc}
\MM\times\CP^1 & \longrightarrow & \hat{Z}\\
(u,z)          & \longmapsto     & u(z)
\end{array}\]
is proper and of degree~$1$. It restricts to a proper degree~$1$ map
\[ \ev\co\MM\times\C\longrightarrow Z.\]
\end{prop}

\begin{proof}
Let $(u^{\alpha})_{1\leq\alpha\leq N}$ be a stable map
in the sense of \cite{mcsa04} that arises as the Gromov-limit of a
sequence $(u_{\nu})$ of spheres in~$\MM$. We need to show that $N=1$, so that
$u_{\nu}\rightarrow u^1$ as a $C^{\infty}$-limit. The claim
about the first evaluation map then follows from Lemma~\ref{lem:spheres}
and the observation following the definition of~$\MM$,
which say that the non-compact ends of $\hat{Z}$ are
foliated by holomorphic spheres $\{v\}\times\CP^1$, and no other spheres 
intersect these ends.

By positivity of intersections \cite[Proposition~7.1]{cimo07}
and (M1) we may assume that $u^1\bullet H=1$ and $u^j\bullet H=0$ for
$j=2,\ldots, N$. Then Lemma~\ref{lem:spheres}~(ii) tells us that
the $u^j$ are contained entirely in $\Int(W)$ for $j=2,\ldots, N$,
but $W$ does not contain any non-constant holomorphic spheres.

From positivity of intersection and (M2) we conclude that
$u^{-1}(H)=\{\infty\}$ for $u\in\MM$, so the evaluation map
restricts to $\C\subset\CP^1$ as claimed.
\end{proof}

\subsection{A homology epimorphism}
\label{subsection:hom-epi}
From Proposition~\ref{prop:ev} we now deduce crucial homological
information.

\begin{prop}
\label{prop:hom-epi}
The induced homomorphism
\[ \ev_*\co H_k(\MM\times\C)\longrightarrow H_k(Z)\]
is surjective in all degrees~$k$.
\end{prop}

\begin{proof}
Write $D^2_R\subset\C$ for the closed $2$-disc of radius~$R$.
By Lemma~\ref{lem:spheres}, for $R$ sufficiently large
we have
\[ \ev\bigl(\MM\times(\C\setminus\Int(D^2_R))\bigr)\subset Z\setminus W.\]

Write $\MM'\subset\MM$ for the truncated moduli space
obtained by cutting off the non-compact end of $\MM$
consisting of spheres $\{v\}\times\CP^1$ with $\psi_V(v)>(\log 2)/4$.
Then $\MM\times\C$
strongly deformation retracts to the compact manifold
(with boundary) $P:=\MM'\times D^2_R$.

As observed after the construction of $Z$ and~$\hat{Z}$,
the compact manifold $W$ (with boundary~$M$) is a strong
deformation retract of~$Z$. By pre- and postcomposing
$\ev$ with the respective deformation retraction, we obtain
a map of pairs
\[ f\co (P,\partial P) \longrightarrow (W,\partial W).\]
Since the degree of $\ev$ can be computed at any regular value $w$
in the interior of $W$, and neither $w$ nor the discrete set of points
$\ev^{-1}(w)\subset P$
is affected by the deformation retractions, the map $f$ is likewise
of degree~$1$.
This degree can now be interpreted homologically; $\deg (f)=1$
says that the fundamental cycle $[P]\in H_{2n}(P,\partial P)$
is mapped to the fundamental cycle $[W]\in H_{2n}(W,\partial W)$.

It then follows that the shriek homomorphism $f_!\co H_k(W)\rightarrow
H_k(P)$ is a right inverse for $f_*\co H_k(P)\rightarrow H_k(W)$,
since the composition $f_*f_!$ is simply multiplication by the
homological degree, see~\cite[Proposition~VI.14.1]{bred93}.

Hence, $f_*\co H_k(P)\rightarrow H_k(W)$ is surjective in all degrees, and
the same is true for $\ev_*$, as we have merely passed to deformation
retracts.
\end{proof}
\subsection{Proof of Theorem~\ref{thm:homology}~(a)}
\label{subsection:proof-homology-a}
By condition (M2) we have the commutative diagram
\begin{diagram}
H_k\bigl(\MM\times\{1\}\bigr) & \rTo^{\ev_*} & H_k\bigl(V\times\{1\}\bigr) \\
\dTo^{i_*}                    &              & \dTo_{j_*} \\
H_k(\MM\times\C)              & \rTo^{\ev_*} & H_k(Z),
\end{diagram}
where the vertical homomorphisms are induced by inclusion.
Since $i_*$ is an isomorphism and $\ev_*$ at the bottom
is surjective by Proposition~\ref{prop:hom-epi}, the homomorphism
$j_*$ is likewise surjective.

The Stein manifold $V$ has the homotopy type of an $\ell_0$-dimensional
complex. It follows that
\[ H_k(Z)=0\;\;\;\text{for $k\geq\ell_0+1$.}\]
The same homological vanishing result holds for the
deformation retract $W$ of~$Z$. This means that the homological
dimension of $W$ can be at most that of the subcritical
filling~$W_0$. Beware that, \emph{a priori}, the homotopical
dimension, i.e.\ the smallest dimension of a CW complex homotopy
equivalent to~$W$, might well be larger.

\begin{lem}
\label{lem:relative}
The relative homology group $H_k(W,M)$ vanishes for $k\leq 2n-1-\ell_0$.
\end{lem}

\begin{proof}
Write $FH_*$ and $TH_*$ for the free and the torsion part, respectively,
of a homology group $H_*$. By Poincar\'e duality and the universal
coefficient theorem we have
\[ H_k(W,M)\cong H^{2n-k}(W)\cong FH_{2n-k}(W)\oplus
TH_{2n-k-1}(W).\]
As we have shown, the homological dimension of $W$ is at most equal
to~$\ell_0$. The lemma follows for $k<2n-1-\ell_0$.

For $k=2n-1-\ell_0$, it remains to show that $H_{\ell_0}(W)$ is a torsion-free
group. Since $W_0$ has the homotopy
type of an $\ell_0$-dimensional complex, the homology group
$H_{\ell_0}(W_0)$ is torsion-free.
From the homology exact sequence of the pair $(W_0,M)$
we see with Poincar\'e duality and the universal coefficient theorem
that $H_{\ell_0}(M)\cong H_{\ell_0}(W_0)$, so $H_{\ell_0}(M)$ is likewise
torsion-free.

For $\ell_0< n-1$, the cohomology exact sequence of the pair $(W,M)$,
together with the universal coefficient theorem, reduces to
\[ 0\longrightarrow H^{2n-1-\ell_0}(M)\longrightarrow H^{2n-\ell_0}(W,M)
\longrightarrow 0,\]
and hence $H_{\ell_0}(M)\cong H_{\ell_0}(W)$ by Poincar\'e duality. For
$\ell_0=n-1$ we observe that $H_n(W,M)\cong TH_{n-1}(W)$ by
Poincar\'e duality and the universal coefficient theorem,
and then the relevant part of the homology exact sequence of the
pair $(W,M)$ becomes
\[ 0\longrightarrow TH_{n-1}(W)\longrightarrow H_{n-1}(M).\]
Since $H_{n-1}(M)=H_{\ell_0}(M)$ is torsion-free, this
implies $TH_{n-1}(W)=0$.
\end{proof}

\begin{rem}
Alternatively, one sees from the commutative diagram
that $H_k(W;\F)$ vanishes for $k\geq \ell_0+1$ and all fields~$\F$.
One then deduces the vanishing of the relative homology
groups in the lemma over $\F$ with Poincar\'e and Kronecker duality.
Since this vanishing holds for any field $\F$, it must also hold over~$\Z$.
\end{rem}

With this lemma, Theorem~\ref{thm:homology}~(a) is an immediate
consequence of the homology exact sequence of the pair $(W,M)$.
\subsection{Proof of Theorem~\ref{thm:homology}~(b)} The image
of $\MM\times\{1\}$ under the evaluation map lies in $V\times\{1\}$,
which we may regard as a subset of $(V\times\C)\setminus\Int(W_0)$.
This gives us the commutative diagram
\begin{diagram}
\MM\times\{1\} & \rTo^{\ev} & (V\times\C)\setminus\Int(W_0) \\
\dTo^i         &            & \dTo_j \\
\MM\times\C    & \rTo^{\ev} & Z.
\end{diagram}
The inclusion map $i$ induces an isomorphism on fundamental
groups, and the proper degree~$1$ map $\ev$ at the bottom,
an epimorphism. The latter follows from the fact that
$\ev\co\MM\times\C\rightarrow Z$ factors through
the covering of $Z$ corresponding to the characteristic
subgroup $\ev_*\bigl(\pi_1(\MM\times\C)\bigr)\subset\pi_1(Z)$.
The covering map has degree equal to the index of
$\ev_*\bigl(\pi_1(\MM\times\C)\bigr)$ in $\pi_1(Z)$,
and the degree of proper maps is multiplicative under composition.

Hence, $j$ also induces an epimorphism on fundamental groups. Up
to deformation retraction, $j$ may be regarded as the inclusion map
$M\subset W$.

If $\pi_1(M)$ is abelian, so is its image $\pi_1(W)=j_*\bigl(\pi_1(M)\bigr)$,
in which case these fundamental groups equal the respective first homology
group. Then, by part (a) of the theorem, the inclusion $M\subset W$ is
$\pi_1$-isomorphic.
\subsection{A homology epimorphism for Liouville fillings}
\label{subsection:Liouville}
A close inspection of the proof of Theorem~\ref{thm:homology} shows
that the requirement that the filling $(W_0,\omega_0)$ carry
a Stein structure is not essential; the crucial point was the
$(W_0,\omega_0)$ is a split manifold $V_0\times D^2\subset V\times\C$,
where the `end' $V\setminus V_0$ is of the form $\R^+_0\times\partial V_0$
and has suitable convexity properties for an analogue
of Lemma~\ref{lem:spheres} to hold.

Recall the following definitions from \cite[Chapter~11]{ciel12}.

\begin{defn}
A \textbf{Liouville manifold}
is an exact symplectic manifold $(V,\rmd\lambda)$ such that
the corresponding Liouville vector field~$Y$, defined by $i_Y\rmd\lambda
=\lambda$ is complete, and $V$ has an exhaustion by compact domains
with smooth boundaries, along which $Y$ is outward pointing.
Such compact domains in $(V,\rmd\lambda)$ are called
\textbf{Liouville domains}.

Let $M$ be a compact, connected manifold with a cooriented
contact structure~$\xi$.
A compact symplectic manifold $(W,\omega)$ with boundary is called a
\textbf{Liouville filling} of $(M,\xi)$ if $\partial W=M$
as oriented manifolds and there is a global primitive $1$-form $\lambda$
for $\omega$ such that $\lambda|_{TM}$ is a contact form for~$\xi$.
\end{defn}

\begin{rem}
\label{rem:Liouville}
Whenever we restrict attention to a compact subset of
the Liouville manifold $(V,\rmd\lambda)$ --- for instance, when we
consider compactly supported symplectic isotopies --- we
may assume without loss of generality that the Liouville manifold
is of finite type, that is, it looks like the completion
\[ (V_0,\rmd\lambda)\cup_{\partial V_0}\bigl(\R_0^+\times\partial V_0,
\rmd(\rme^t\lambda_0)\bigr)\]
of a Liouville domain $(V_0,\rmd\lambda)$, where $\lambda_0:=
\lambda|_{T\partial V_0}$.
\end{rem}

\begin{thm}
\label{thm:homology-Liouville}
Let $(M,\xi)$ be a $(2n-1)$-dimensional compact, connected
contact manifold admitting 
a split symplectic filling $V_0\times D^2$ (with corners
rounded), where $(V_0,\rmd\lambda)$ is a Liouville domain.
Let $(W,\omega)$ be any symplectically aspherical filling
of $(M,\xi)$. Then there is a surjective homomorphism
\[ H_k(V_0)\longrightarrow H_k(W)\]
in all degrees $k\geq 0$.

If $V_0$ has the homotopy type of an $\ell_0$-dimensional complex
with $\ell_0\leq n-1$, then the other conclusions of Theorem~\ref{thm:homology}
hold true as well.
\end{thm}

\begin{proof}
In the proof of Theorem~\ref{thm:homology} above, we now
take $V$ to be the completion of~$V_0$. In the analogue
of Lemma~\ref{lem:spheres}, we work with $\End:=(V\times\CP^1)\setminus
(V_0\times D^2_{\rho})$ for a disc $D^2_{\rho}\subset\C\subset\CP^1$
of suitable radius~$\rho$. For (i), it is sufficient to know
that $V$ is exact symplectic. For (ii) --- assuming
$u(\CP^1)$ intersects $\End$ but not the hypersurface~$H$ ---
one applies the maximum principle to the $V$-component of~$u$,
if that component intersects $V\setminus V_0$, or to the $\C$-component,
if that hits $\C\setminus D^2_{\rho}$. The proof of (iii) is analogous.
Then the argument as in Section~\ref{subsection:proof-homology-a}
yields the claimed homology epimorphism.

Under the additional homotopical assumption on~$V_0$, Lemma~\ref{lem:relative}
still holds, and then one concludes as in the proof of
Theorem~\ref{thm:homology}.
\end{proof}
\section{Applications of Theorem~\ref{thm:homology}}
\label{section:applications}
\subsection{Fillings of the standard sphere}
\label{subsection:sphere}
Theorem \ref{thm:EFM} from the Introduction is contained in
Theorem~\ref{thm:homology}. Indeed, $(S^{2n-1},\xist)$ has
a Stein filling given by the unit ball in $\C^n$, so
we have $\ell_0=0$ in the notation of Theorem~\ref{thm:homology}.
This theorem then says that any other symplectically aspherical
filling $(W,\omega)$ of $(S^{2n-1},\xist)$ is a simply connected
homology ball of dimension $2n\geq 6$, and hence diffeomorphic to the
standard ball by Proposition~A on page 108 of~\cite{miln65}.

\subsection{Liouville fillings}
Using relative de Rham theory, we now derive a general
property of the fillings that can arise in the situation of
Theorem~\ref{thm:homology}.

\begin{prop}
Let $(M,\xi)$ be a contact manifold that admits a subcritical Stein
filling. Then any symplectically aspherical filling
$(W,\omega)$ of $(M,\xi)$ is a Liouville filling.
\end{prop}

\begin{proof}
Write $i\co M\rightarrow W$ for the inclusion map. The $k$th relative
de Rham chain group of the pair $(W,M)$ is given by
$\Omega^k(i)=\Omega^k(W)\oplus\Omega^{k-1}(M)$, and the
differential is given by $\rmd(\eta,\mu)=(\rmd\eta, i^*\eta-\rmd\mu)$,
see~\cite[p.~78]{botu82}.

By assumption on $(W,\omega)$ being a symplectic filling of $(M,\xi)$,
there is a Liouville vector field $Y$ defined near $M$ and
pointing transversely outwards such that the $1$-form $\lambda:=i_Y\omega$,
defined near $M$, restricts to a contact form $\alpha:=i^*\lambda$ for~$\xi$.
We want to show that we can find a global primitive $\lambda$ with this
property.

The pair $(\omega,\alpha)$ is closed, since $\rmd(\omega,\alpha)=
(\rmd\omega,i^*\omega-\rmd\alpha)=(0,0)$, so this pair defines
a class in $H^2_{\mathrm{dR}}(W,M)\cong H_{2n-2}(W;\R)$, where as before
we take the dimension of $W$ to be $2n$.
By Theorem~\ref{thm:homology}, this last homology group vanishes,
so the pair $(\omega,\alpha)$ is actually exact. This means we can
find a $1$-form $\mu$ on $W$ and a smooth function $f$ on $M$ such that
$(\omega,\alpha)=(\rmd\mu,i^*\mu-\rmd f)$.

Extend $f$ to a smooth function $F$ on $W$ and set $\lambda:=\mu-\rmd F$.
Then $\rmd\lambda=\omega$ and $i^*\lambda=i^*\mu-\rmd f=\alpha$.
\end{proof}

\subsection{A result of Oancea--Viterbo}
\label{section:OV}
The following is Theorem~2.6 of \cite{oavi12}.

\begin{thm}[Oancea--Viterbo]
\label{thm:OV}
Let $(M_1,\xi_1)$ be a compact, connected contact manifold admitting
an embedding into a subcritical Stein manifold as a hypersurface of
contact type. Let $(W_1,\omega_1)$ be any symplectically aspherical
filling of $(M_1,\xi_1)$ satisfying one of the following conditions:
\begin{itemize}
\item[(i)] $H_2(W_1,M_1)=0$;
\item[(ii)] $M_1$ is simply connected.
\end{itemize}
Then the homomorphism
\[ H_k(M_1)\longrightarrow H_k(W_1)\]
induced by the inclusion $M_1\subset W_1$ is surjective in all degrees~$k$.
\end{thm}

\begin{rem}
(1) The indexing $1$ is used here merely to avoid notational
confusion when we reprove this result below.

(2) In case (i), symplectic asphericity of any filling $(W_1,\omega_1)$
of $(M_1,\xi_1)$ is a direct consequence of the homological assumption
$H_2(W_1,M_1)=0$, for the homomorphism $H_2(M_1)\rightarrow H_2(W_1)$
induced by inclusion is then surjective, and $\omega_1$ is exact near~$M_1$.

(3) If $W_1$ is itself Stein (not \emph{a priori} subcritical), then
condition (i) is automatic for $\dim W_1=2n\geq 6$, since
$W_1$ then has the homotopy type of an $n$-dimensional complex,
and so $H_2(W_1,M_1)\cong H^{2n-2}(W_1)$ is zero for $2n-2>n$.
\end{rem}

We now prove Theorem~\ref{thm:OV}.
As in Section~\ref{subsection:completion}
we have a contact type embedding of $(M_1,\xi_1)$ into a split
Stein manifold $V\times\C$. Then $M_1$ separates $V\times\C$ into
a compact and a non-compact component. By assumption, there
is a Liouville vector field for the symplectic form on $V\times\C$,
defined near and transverse to~$M_1$;
by \cite[Theorem~3.4]{geze12}, this Liouville vector field
points out of the compact component.

Choose a level set $M$ of
the plurisubharmonic function $\psi$ on $V\times\C$ for a level
so large that $M_1$ is contained in its sublevel set. Equip $M$
with the contact structure $\xi$ induced by the Stein structure.
Then the compact region between $M_1$ and $M$ defines a symplectic
cobordism $(W_2,\omega_2=\rmd\lambda_2)$ from $(M_1,\xi_1)$ to $(M,\xi)$.

Now, given a symplectically aspherical filling $(W_1,\omega_1)$ of
$(M_1,\xi_1)$, we can glue it along this contact boundary to the cobordism
$(W_2,\omega_2)$, resulting in a filling $(W,\omega)$
of the contact manifold $(M,\xi)$. The corresponding sublevel set of $\psi$
defines a subcritical Stein filling of $(M,\xi)$, so the first
assumption of Theorem~\ref{thm:homology} is satisfied.

\begin{lem}
\label{lem:i-ii-aspherical}
Under either of the assumptions (i) or (ii)
in Theorem~\ref{thm:OV}, $(W,\omega)$ is symplectically aspherical.
\end{lem}

\begin{proof}
Let $S$ be a $2$-sphere in $W$ (i.e.\ a map $\phi\co S^2\rightarrow W$,
without loss of generality assumed to be smooth, with
image~$S$). We need to show that $\int_S\omega:=\int_{S^2}\phi^*\omega=0$.
The part $(W_1,\omega_1)$ of $(W,\omega)$ is symplectically aspherical by
assumption; $(W_2,\omega_2)$ is symplectically aspherical since
$\omega_2=\rmd\lambda_2$ is
exact. So $S$ cannot be entirely contained in only one of these two parts.

Make $S$ transverse to $M_1$ and write $S_1,S_2$ for the parts
of $S$ contained in $W_1,W_2$, respectively.

(i) If $H_2(W_1,M_1)=0$, the relative cycle $S_1$ represents the zero class,
so there is a relative $3$-chain $C$ with $\partial C=S_1\cup\Sigma$,
where $\Sigma$ is a $2$-chain in $M_1$ with $\partial\Sigma=-\partial S_1$.
It follows that
\[ \int_{S_1}\omega+\int_{\Sigma}\omega=\int_{\partial C}\omega=
\int_C\rmd\omega=0,\]
and further
\[ \int_{S_1}\omega=-\int_{\Sigma}\omega=-\int_{\Sigma}\rmd\lambda_2=
-\int_{\partial\Sigma}\lambda_2=-\int_{\partial S_2}\lambda_2=
-\int_{S_2}\omega_2.\]
This gives $\int_S\omega=0$.

(ii) If $M_1$ is simply connected, then $S_1$ and $S_2$ can be closed
off to $2$-spheres $\hat{S}_1,\hat{S}_2$ by $2$-discs in~$M_1$. Then
\[ \int_S\omega=\int_{S_1}\omega_1+\int_{S_2}\omega_2=
\int_{\hat{S}_1}\omega_1+\int_{\hat{S}_2}\omega_2=0\]
by the symplectic asphericity of $(W_1,\omega_1)$ and $(W_2,\omega_2)$.
\end{proof}

Thanks to this lemma, we can use the information
from Theorem~\ref{thm:homology}~(a) as input in the Mayer--Vietoris
sequence of the decomposition $W=W_1\cup W_2$. For $k\geq n$
we have $H_k(W)=0$, and hence the exact sequence
\[ H_k(M_1)\longrightarrow H_k(W_1)\oplus H_k(W_2)\longrightarrow 0;\]
in particular, the homomorphism $H_k(M_1)\rightarrow H_k(W_1)$ is surjective.

For $k\leq n-1$ we still have that the homomorphism $H_k(M)\rightarrow
H_k(W)$ is surjective. Consider the commutative diagram
\begin{diagram}
         &                  &          &        &          & 
                        & H_k(M) \\
         &                        &          &        &          &
   \ldTo^{i_M}    & \dOnto_{j_M} \\  
H_k(M_1) & \rTo^{(i_1,i_2)} & H_k(W_1) & \oplus & H_k(W_2) & 
   \rTo^{j_1-j_2} & H_k(W),
\end{diagram}
where all homomorphisms are induced by inclusion maps,
the row is exact, and the vertical homomorphism is  surjective. We want to
show that the homomorphism $i_1$ is surjective.
Given a class $a_1\in H_k(W_1)$, set $a:=j_1(a_1)\in H_k(W)$.
Choose $A\in H_k(M)$ with $j_M(A)=a$, and set $a_2:=i_M(A)\in H_k(W_2)$,
so that $j_2(a_2)=a$. It follows that
$(a_1,a_2)\in H_k(W_1)\oplus H_k(W_2)$ maps to zero under
$j_1-j_2$, and so this pair lies in the image of $(i_1,i_2)$.
This shows that $i_1$ is an epimorphism, which completes the proof
of Theorem~\ref{thm:OV}.
\subsection{Extension of the Oancea--Viterbo result to $\pi_1$}
Using Theorem~\ref{thm:homology}~(b), we can formulate a result
analogous to Theorem~\ref{thm:OV} for the fundamental group.

\begin{prop}
Under the assumptions of Theorem~\ref{thm:OV}, the normal
subgroup $\NN\bigl(i_{1\#}(\pi_1(M_1))\bigr)$ generated by the image
of $\pi_1(M_1)$  in $\pi_1(W_1)$ equals the full group $\pi_1(W_1)$.
\end{prop}

\begin{proof}
By Lemma~\ref{lem:i-ii-aspherical} we may apply
Theorem~\ref{thm:homology}~(b) to the symplectic manifold
$(W=W_1\cup W_2,\omega)$ constructed in the preceding section.
Thus, we know that the homomorphism $\pi_1(M)\rightarrow \pi_1(W)$ is
surjective. This homomorphism factors through $\pi_1(W_2)$, so
$\pi_1(W_2)\rightarrow\pi_1(W)$ is likewise surjective. Here
all fundamental groups are taken with a base point $*\in M$,
but the last epimorphism continues to hold when we switch to a base point
$*_1\in M_1$. From now on, this base point $*_1$ will be understood.

By Seifert--van Kampen, the fundamental group $\pi_1(W)$ is
an amalgamated product
\begin{diagram}
           &       & \pi_1(W_1) &         &  \\
           & \ruTo &            & \rdTo   &  \\
\pi_1(M_1) &       & \rTo       &         & \pi_1(W)=\pi_1(W_1)*_{\pi_1(M_1)}
                                            \pi_1(W_2).\\
           & \rdTo &            & \ruOnto &  \\
           &       & \pi_1(W_2) &         &
\end{diagram}

Form a $CW$ complex $W_2'$ from $W_2$ by attaching discs to loops
in $W_1\setminus M_1$ freely homotopic to a set of generators of
$\pi_1(W_2)$. Then $\pi_1(W_2')=\{1\}$ and $W_1\cap W_2'=M_1$.
Moreover, since the homomorphism $\pi_1(W_2)\rightarrow\pi_1(W)$ is
surjective, the space $W_1\cup W_2'$ is simply connected.
With Seifert--van Kampen we have
\[ \{1\}=\pi_1(W_1\cup W_2')=\pi_1(W_1)*_{\pi_1(M_1)}\{1\}=
\pi_1(W_1)/\NN\bigl(i_{1\#}(\pi_1(M_1))\bigr).\qedhere\]
\end{proof}
\subsection{Milnor fillable contact structures}
Let $f\co(\C^{n+1},0)\rightarrow(\C,0)$, $n\geq 3$,
be a complex polynomial function
with an isolated singularity at the origin, i.e.\ an isolated
common zero of the partial derivatives $\partial_{z_j}f$,
$j=0,\ldots,n$. Choose $\varepsilon>0$ sufficiently small
so that the $\varepsilon$-disc around the origin in $\C^{n+1}$
does not contain any further singularities of~$f$. Then the
link $M_f$ of the singularity,
\[ M_f:=S^{2n+1}_{\varepsilon}\cap\{f=0\},\]
is a $(2n-1)$-dimensional manifold with contact structure
$\xi_f$ given by the complex tangencies,
\[ \xi_f:=TM_f\cap\rmi(TM_f),\]
see~\cite{cnp06}. For $\delta\in\C^*$ with
$|\delta|$ sufficiently small, the smoothing
\[ W_f:=D^{2n+2}_{\varepsilon}\cap\{f=\delta\}\]
with its canonical Stein symplectic structure $\omega_f$ is,
by Gray stability of contact structures, a Stein filling of
$(M_f,\xi_f)$. We call this Stein manifold, which is unique up
to deformation equivalence, the \emph{Milnor filling} of $(M_f,\xi_f)$.
(In \cite{cnp06}, that name refers to the singular filling.)

The \emph{Milnor number} $\mu$ is the degree of the map
\[ \begin{array}{ccc}
S^{2n+1}_{\varepsilon} & \longrightarrow & S^{2n+1}_{\varepsilon}\\
\bfz                   & \longmapsto     & g(\bfz)/|g(\bfz)|,
\end{array}\]
where $g:=(\partial_{z_0}f,\ldots,\partial_{z_n}f)$.
This number is always non-negative, and it equals zero precisely
when the origin is actually a non-singular point of~$f$.

\begin{prop}
Suppose the contact manifold $(M_f,\xi_f)$ admits a subcritical Stein
filling $(W_0,\omega_0)$. Then the following holds:
\begin{itemize}
\item[(i)] $W_0$ and $W_f$ are diffeomorphic to the disc $D^{2n}$.
\item[(ii)] The  Stein structure on the Milnor filling $W_f$ is the
standard Stein structure on the disc.
\item[(iii)] $(M_f,\xi_f)$ is contactomorphic to $(S^{2n-1},\xist)$.
\end{itemize}
\end{prop}

\begin{proof}
The map
\[ f/|f|\co S^{2n+1}_{\varepsilon}\setminus M_f\longrightarrow S^1\]
is a locally trivial fibration, the closure of whose fibre
(the so-called \emph{Milnor fibre}) is diffeomorphic to~$W_f$, see
\cite[Theorem~5.11]{miln68}. Then, by \cite[Theorem~6.6]{miln68},
which applies for $n\geq 3$, and \cite[Theorem~7.2]{miln68},
the Milnor filling $W_f$
is diffeomorphic to a $2n$-dimensional handlebody
obtained by attaching $\mu$ handles of index $n$ to $D^{2n}$;
in particular, it is homotopy equivalent to a bouquet of $\mu$
spheres of dimension~$n$.

Now, by Theorem~\ref{thm:homology}~(a), the assumption on the existence
of a subcritical Stein filling $(W_0,\omega_0)$ implies that
$H_k(W_f)=0$ for $k\geq n$, which forces $\mu=0$. Thus,
$W_f$ is diffeomorphic to~$D^{2n}$.

In particular, $M_f$ is diffeomorphic to $S^{2n-1}$. For the argument that
follows, it would be enough to know that $M_f$ is simply connected,
which holds by \cite[Theorem~5.2]{miln68} and our assumption $n\geq 3$.
Theorem~\ref{thm:homology}~(b), applied to~$W_0$, tells us that
$W_0$ is likewise simply connected. Part (a) of the theorem,
applied to both $W_f$ and $W_0$, tells us that $W_0$ is a homology ball.
Then, as in Section~\ref{subsection:sphere}, we conclude that
$W_0$ is also diffeomorphic to~$D^{2n}$. This proves~(i).

In order to prove (ii), we observe that, because of $\mu=0$,
the origin is a non-singular point of~$f$.
By relabelling the coordinates and multiplying $f$ by a suitable complex
constant, we may assume that $\partial_{z_0}f(0)=1$. Then, for $\varepsilon>0$
sufficiently small, the linear interpolation between $f$ and the
function $\bfz\mapsto z_0$ does not develop any singularity
in $D^{2n}_{\varepsilon}$. This interpolation provides the
Stein deformation of $(W_f,\omega_f)$ to the standard Stein structure
on~$D^{2n}$.

Statement (iii) is an immediate consequence of~(ii).
\end{proof}

For applications of Theorem~\ref{thm:OV} to Milnor fillable contact
manifolds see~\cite[Section~6]{oavi12}.
\subsection{Distinguishing contact structures}
The homological information in Theorem~\ref{thm:homology}
gives a simple criterion to distinguish contact structures $\xi,\xi'$
on a given manifold~$M$. Suppose $(M,\xi)$ is subcritically Stein fillable
with a Stein manifold of homotopical dimension~$\ell_0$, or Liouville
fillable as in Theorem~\ref{thm:homology-Liouville} (including the
homotopical assumption there), and $(M,\xi')$
has a symplectically aspherical filling of homotopical dimension
greater than~$\ell_0$, then $\xi$ and $\xi'$ are not diffeomorphic.

We illustrate this with two simple examples. With $\lambda_Q$ we
denote the canonical Liouville $1$-form on the cotangent bundle
of a manifold~$Q$.

\begin{ex}
\label{ex:distinguish}
(1) The unit sphere bundle $S(T^*S^2\oplus\C)$ of the stabilised cotangent
bundle of $S^2$ is diffeomorphic to $S^3\times S^2$, and it
inherits a contact structure $\xi$
from the symplectic structure $\rmd\lambda_{S^2}+\rmd x\wedge\rmd y$
on the unit disc bundle $D(T^*S^2\oplus\C)$. By \cite[Example~6.2.8]{geig08},
we can think of $(S^3\times S^2,\xi)$ and its filling
as the result of attaching a symplectic $2$-handle to the standard
symplectic $6$-ball along a standard isotropic $S^1\subset (S^5,\xist)$.

On the other hand, the standard contact structure $\xi'$ on
$ST^*S^3\cong S^3\times S^2$ with symplectically aspherical filling
$(S^3\times D^3,\rmd\lambda_{S^3})$ is the result of attaching
a symplectic $3$-handle along a Legendrian $S^2\subset (S^5,\xist)$.

Both contact structures have vanishing first Chern class, so their
underlying almost contact structures are
homotopic~\cite[Proposition~8.1.1]{geig08},
but by Theorem~\ref{thm:homology} the contact structures are not diffeomorphic.
See also \cite[Example~1.9]{gnw16} and, for the handle
descriptions, the discussion in \cite[p.~1196]{geig97}.

(2) Likewise, the contact structures on $S^7\times S^6$ coming
from the description as $S(T^*S^6\oplus\C)$ and $ST^*S^7$,
respectively, are not diffeomorphic. This also follows from
\cite[Corollary~1.18]{cfo10}, whose proof employs
Rabinowitz Floer homology.
\end{ex}

\begin{rem}
Given a (stabilised) cotangent bundle,
there is a Stein structure on the unit disc bundle that provides
a Stein filling of the standard contact structure on the
unit sphere bundle. This follows from the explicit description
of a Weinstein structure in~\cite[Example 11.12~(b)]{ciel12}
and the Stein existence theorem~\cite[Theorem~13.5]{ciel12}.
\end{rem}
\subsection{Fillings of unit cotangent bundles}
The examples above can also be regarded as an instance of the
following result.

\begin{prop}
The unit cotangent bundle $(ST^*Q,\ker\lambda_Q)$ of a closed
manifold $Q$ does not admit a subcritical Stein filling.
\end{prop}

\begin{proof}
If it did, this would produce a contradiction to Theorem~\ref{thm:homology},
since the symplectically aspherical filling $(DT^*Q,\rmd\lambda_Q)$
has non-trivial homology in the critical dimension.
\end{proof}

\begin{rem}
For related results see \cite[Corollary~1.18]{cfo10} and
\cite[Corollary~2.2]{alfr12}.
\end{rem}
\section{Dehn--Seidel twists}
\label{section:DS}
Let $L\cong S^{n-1}$ be a Lagrangian sphere in a symplectic manifold
$(V,\omega)$ of real dimension~$2n-2$. By the Weinstein neighbourhood
theorem, cf.~\cite[Theorem~3.33]{mcsa98},
there is a neighbourhood of $L$ symplectomorphic to
a neighbourhood of the zero section in
the cotangent bundle $T^*S^{n-1}$ with its canonical symplectic structure
$\rmd\lambda_{S^{n-1}}$. The inclusion $S^{n-1}\subset\R^n$
gives us a global coordinate description
of the Liouville $1$-form $\lambda_{S^{n-1}}$.
In terms of Cartesian coordinates $(\bfq ,\bfp) \in \R^n\times\R^n$,
the cotangent bundle $T^*S^{n-1}\subset\R^{2n}$ is described by the
equations
\[ \bfq\cdot\bfq=1 \text{ and } \bfq\cdot\bfp=0;\]
then $\lambda_{S^{n-1}}=\bfp\,\rmd\bfq$.

Define a map
\[ \tau\co (T^*S^{n-1},\rmd\lambda_{S^{n-1}})\longrightarrow
(T^*S^{n-1},\rmd\lambda_{S^{n-1}}) \]
as follows. Consider the normalised geodesic flow $\sigma_t$ on
$T^*S^{n-1}\setminus S^{n-1}$ given by
\[ \sigma_t(\bfq,\bfp)=
\left(
\begin{array}{cc}
\cos t        & |\bfp|^{-1} \sin t \\
-|\bfp|\sin t & \cos t
\end{array}
\right)
\binom{\bfq}{\bfp}.\]
Then set
\[ \tau(\bfq,\bfp)=\sigma_{g(|\bfp|)}(\bfq,\bfp),\]
where $r\mapsto g(r)$ is a smooth function that interpolates monotonically
between $\pi$ near $r=0$ and $0$ for large~$r$. For $\bfp=0$ this is read
as $\tau(\bfq,0)=(-\bfq ,0)$.
Then $\tau$ is a
symplectomorphism of $(T^*S^{n-1},\rmd\lambda_{S^{n-1}})$,
equal to the identity for $|\bfp|$ large.
Thus, $\tau$ may be regarded as a symplectomorphism of~$(V,\omega)$,
and it is then called a right-handed \emph{Dehn twist} along $L\subset V$;
for $n=2$ this coincides with the classical notion of a Dehn twist.

These generalised Dehn twists have been introduced
and studied extensively
by Seidel, see \cite[Section~6]{seid99} and~\cite{seid08},
and they are nowadays often referred to as Dehn--Seidel twists.

\begin{rem}
For $n-1$ odd, the model Dehn twist $\tau$ on $T^*S^{n-1}$ is
of infinite order in $\pi_0(\Diff^c(T^*S^{n-1}))$, where $\Diff^c$
denotes the group of compactly supported diffeomorphisms.
For $n-1=2$ or $6$, the order of $\tau$ is two; for other even $n-1$
it is four or eight, see~\cite[p.~3311]{seid14} for a discussion and
references.

One example where $\tau^2$ is even
symplectically trivial (i.e.\ isotopic to the
identity via compactly supported symplectomorphisms)
is the Dehn twist along the anti-diagonal in
$S^2\times S^2$ with the monotone product symplectic structure
(i.e.\ of equal area on the two factors).
Seidel proved that this example is atypical.
For instance, it is shown in \cite{seid08}, based on the
work of Gromov~\cite{grom85}, that the group
of compactly supported symplectomorphisms of $(T^*S^2,\rmd\lambda_{S^2})$
is homotopy equivalent to~$\Z$, generated by $\tau$. Other results of
\cite{seid08} concern the symplectic non-triviality of
$\tau^2$ in dimension four; in~\cite[Section~5]{seid14} it is
shown that the Dehn twist in the cotangent
bundle of any higher-dimensional sphere is of infinite
order symplectically. Seidel's
arguments involve subtle methods from Floer homology.
\end{rem}

We now want to prove Theorem~\ref{thm:DS} from the introduction,
which establishes the symplectic non-triviality of any non-empty
composition of right-handed Dehn twists for a broad class of
symplectic manifolds,
including, in particular, the cotangent bundle
$(T^*S^{n-1},\rmd\lambda_{S^{n-1}})$
\emph{per se}. Thus, let $(V,\rmd\lambda)$ be a Liouville manifold
of dimension at least four.
Write $\Symp^c(V)$ for the group of compactly supported
symplectomorphisms $\phi$ of $(V,\rmd\lambda)$.

\begin{rem}
We want to detect the symplectic non-triviality of compositions of
right-handed Dehn twists in $\Symp^c(V)$. To this end, we shall
be using an argument by contradiction, starting from the assumption
that there is a symplectic isotopy from a given
symplectomorphism of that type to the identity. This allows us,
by Remark~\ref{rem:Liouville}, to assume without loss of contradiction
that all relevant maps and isotopies are supported in the
interior of a Liouville domain $V_0$ whose symplectic completion
is~$V$.
\end{rem}

If $\phi\in\Symp^c(V)$ is \emph{exact}, i.e.\ $\phi^*\lambda-\lambda=\rmd h$
for some smooth function $h\co V\rightarrow\R$
with compact support in~$\Int(V)$, there is a canonical
construction due to Giroux, see~\cite[Theorem~7.3.3]{geig08}
of a contact structure on the open book with page $\Int(V_0)$
and monodromy~$\phi$. We denote this contact manifold by $\Open(V_0,\phi)$.
We want to show that if $\phi$ lies in the identity component
$\Symp^c_0(V)$ of $\Symp^c(V)$, i.e.\ if $\phi$ is isotopic
to the identity via compactly supported but not, \emph{a priori},
exact symplectomorphisms, then the resulting contact manifold
does not depend, up to contactomorphism, on the specific
choice of such~$\phi$.

\begin{prop}
\label{prop:open}
If $\phi\in\Symp^c_0(V)$ is exact, then $\Open(V_0,\phi)$ is
contactomorphic to $\Open (V_0,\mathrm{id})$.
\end{prop}

\begin{proof}
As shown in \cite[Chapter~10]{mcsa98}, there is an exact
sequence
\[ \begin{array}{ccccccccc}
0 & \longrightarrow & \Ham^c(V) & \longrightarrow & \Symp^c_0(V) &
\longrightarrow & H^1_c(V;\R) & \longrightarrow & 0,\\
  &                 &           &                 & \phi         &
\longmapsto     & [\phi^*\lambda-\lambda] & & 
\end{array}\]
where $\Ham^c(V)$ denotes the group of compactly
supported Hamiltonian diffeomorphisms, that is, symplectomorphisms
that are the time-$1$ map of a Hamiltonian isotopy with support in a
compact subset (depending on the isotopy), and
$H^1_c(V;\R)$ denotes the compactly supported de Rham cohomology
of~$V$.

Thus, our assumptions imply that $\phi$ is actually Hamiltonian.
Let $(H_t)_{t\in[0,1]}$ be the time-dependent Hamiltonian function
generating the Hamiltonian isotopy $(\phi_t)_{t\in[0,1]}$
with $\phi_0=\mathrm{id}$ and $\phi_1=\phi$. Write $X_t$
for the corresponding time-dependent Hamiltonian vector field. Then
\begin{equation}
\label{eqn:exact}
\phi_t^*\lambda-\lambda=\rmd h_t,
\end{equation}
where
\[ h_t:=\int_0^t\phi_s^*(\lambda(X_s)-H_s)\,\rmd s,\]
see~\cite[Proposition~9.19]{mcsa98}.

With (\ref{eqn:exact}) one deduces $\Open(V,\mathrm{id})\cong
\Open(V,\phi)$ from the explicit construction of the contact open book
in \cite[Theorem~7.3.3]{geig08}, using
Gray stability.
\end{proof}

The following lemma will allow us to apply this observation to
compositions of Dehn twists. Here the assumption
$\dim V\geq 4$ is used.

\begin{lem}
\label{lem:Dehn-exact}
Any composition of Dehn twists on $(V,\rmd\lambda)$ is an exact
symplectomorphism.
\end{lem}

\begin{proof}
For the model Dehn twist $\tau$ on $(T^*S^{n-1},\rmd\lambda_{S^{n-1}})$,
an explicit function $h$ with
$\tau^*\lambda_{S^{n-1}}-\lambda_{S^{n-1}}=\rmd h$
is described in~\cite{koni05}. The choice of primitive $\lambda_{S^{n-1}}$
for the symplectic form, however, is irrelevant for the exactness
of~$\tau$, as follows from a simple homological consideration.

If $(U,\rmd\mu)$ is any open $(2n-2)$-dimensional symplectic manifold
(without boundary) with
$H^{2n-3}(U;\R)=0$, such as a tubular neighbourhood of a Lagrangian
sphere (for $n\geq 3$), then $H^1_c(U;\R)=0$ by Poincar\'e duality for
compactly supported cohomology~\cite[p.~44]{botu82}. Then,
any compactly supported symplectomorphism of $(U,\rmd\mu)$
is exact, regardless of the choice of primitive $\mu$ for the
symplectic form. Hence, this remains true if $(U,\rmd\mu)$ admits
a (not necessarily exact!) symplectic embedding into a larger symplectic
manifold $(V,\rmd\lambda)$,
and the symplectomorphism is regarded as an automorphism of~$V$.
In particular, Dehn twists on exact symplectic manifolds
of dimension $2n-2\geq 4$ are always exact.

A straightforward calculation shows that the composition of
exact symplectomorphisms is likewise exact.
\end{proof}

\begin{proof}[Proof of Theorem~\ref{thm:DS}]
Arguing by contradiction, we assume that $\phi$ is a non-trivial composition
of right-handed Dehn twists on the Liouville manifold $(V,\rmd\lambda)$
which is symplectically isotopic to the identity, i.e.\
contained in $\Symp^c_0(V)$. 
By Lemma~\ref{lem:Dehn-exact} and Proposition~\ref{prop:open},
the contact open book $\Open(V_0,\phi^N)$ is contactomorphic
to $\Open(V_0,\mathrm{id})$ for any natural number~$N\in\N$.

The contact open book $(M,\xi):=\Open(V_0,\mathrm{id})$ is symplectically
filled by $V_0\times D^2$ (after rounding corners inside $V\times\C$).
This places us in the situation of Theorem~\ref{thm:homology-Liouville}.

Now we use the specific nature of $\phi$ as a composition of right-handed
Dehn twists. By \cite[Lemma~4.2]{koer10}, a Lagrangian sphere $L$ in the
page of a contact open book may be assumed to be Legendrian
in the open book. Then, by \cite[Theorem~4.4]{koer10},
composing the monodromy of the given open book with the right-handed
Dehn twist along $L$ is equivalent to performing a Weinstein
surgery on the open book along~$L$; see~\cite{geig12} for a simpler
description in the $3$-dimensional situation.

This gives us a Liouville (and hence symplectically aspherical)
filling $(W_N,\omega_N)$ of $\Open(V_0,\phi^N)\cong (M,\xi)$ for any
$N\in\N$. By Theorem~\ref{thm:homology-Liouville}, we have,
in particular, a homology epimorphism $H_n(V_0)\rightarrow H_n(W_N)$.

On the other hand, $W_N$ is obtained from $V_0\times D^2$
by attaching $kN$ handles of index~$n$, where $k$ is the
number of Dehn twists in the composition~$\phi$.
Think of $W_N$ as decomposed into $V_0\times D^2$ and the handles,
with intersection given by the disjoint neighbourhoods,
each diffeomorphic to $S^{n-1}\times D^n$, of $kN$
attaching spheres in $\partial(V_0\times D^2)$. The relevant part
of the Mayer--Vietoris sequence,
\[ H_n(W_N)\longrightarrow H_{n-1}(\sqcup_{kN}S^{n-1})\longrightarrow
H_{n-1}(V_0),\]
gives us the estimate
\[ b_n(W_N)+b_{n-1}(V_0)\geq kN\]
on Betti numbers.
Together with the epimorphism $H_n(V_0)\rightarrow H_n(W_N)$
we have
\[ b_n(V_0)+b_{n-1}(V_0)\geq kN\]
for all $N\in\N$, which is a contradiction.
\end{proof}

\begin{rem}
With the notation from this section, Theorem~\ref{thm:DS}
can be rephrased as saying that for any non-trivial
composition $\phi$ of right-handed Dehn twists on a Liouville
manifold $(V,\rmd\lambda)$,
the class of $\phi$ is of infinite order in $\pi_0(\Symp^c(V))$.
\end{rem}

\begin{rem}
\label{rem:DS}
For the iteration of a single Dehn twist on a Liouville domain,
Theorem~\ref{thm:DS}
has been proved by Uljarevi\'c \cite[Corollary~5.6.3]{ulja16},
using Floer-theoretic methods. The analogue of the
theorem for iterations of a single fibred or fractional Dehn twist,
under various technical assumptions, 
are proved in \cite[Theorem~B]{cdvk14}, \cite[Corollary~1.2]{cdvk16}
and \cite[Corollary~1.4]{ulja14}. The papers \cite{cdvk14,cdvk16}
also build on the idea to use the iterated (fractional or fibred) Dehn twist
as the monodromy of an open book. The infinite order of
the Dehn twist is then established by considering the mean
Euler characteristic in symplectic homology, or by studying filling
obstructions.
\end{rem}
\section{Proof of Theorem~\ref{thm:filling}}
\label{section:filling}
Let $(M,\xi)$ be a simply connected contact manifold of dimension
at least five,
with a subcritical Stein filling $(W_0,\omega_0)$, and
let $(W,\omega)$ be any other symplectically aspherical
filling of $(M,\xi)$. Our aim is to show that $W$ must be
diffeomorphic to~$W_0$.

As in the proof of Theorem~\ref{thm:homology}, we may think of $W$
as a level set of a plurisubharmonic potential on a split
Stein manifold $V\times\C$, and of $W_0$ as the
corresponding sublevel set.

\begin{figure}[h]
\labellist
\small\hair 2pt
\pinlabel $V$ [b] at 565 217
\pinlabel $\C$ [l] at 288 421
\pinlabel $W$ at 175 197
\pinlabel $W_0$ at 175 324
\pinlabel $W_1$ at 95 83
\pinlabel $M_0$ [l] at 432 340
\pinlabel $M_1$ [l] at 541 340
\pinlabel $V_1$ [b] at 514 396
\pinlabel $V_0$ [t] at 388 377
\endlabellist
\centering
\includegraphics[scale=0.55]{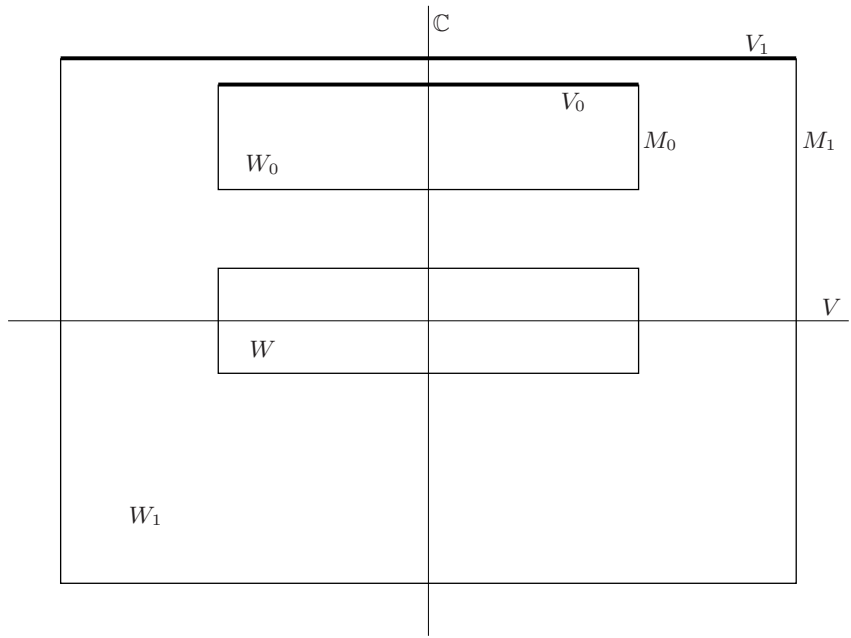}
  \caption{The cobordism $X=W_1\setminus\Int(W_0)$.}
  \label{figure:sq-cobordism}
\end{figure}

Consider the schematic picture shown in Figure~\ref{figure:sq-cobordism}.
From now on, the argument is essentially topological. This allows us
to think of $W_0$ as $V_0\times D^2$, where $V_0$ is a Stein
domain with symplectic completion~$V$ (in the sense of
Remark~\ref{rem:Liouville}).

We build a manifold $W_1$ as follows. Add a (sufficiently large)
collar neighbourhood
to $W_0$, i.e.\ pass to a higher sublevel set of the plurisubharmonic
potential, with boundary~$M_1\cong M$.
There is a topological
copy of $W_0$ inside this neighbourhood (and it is this which
is shown in Figure~\ref{figure:sq-cobordism}), disjoint from the
original one,
simply given by translation in the $\C$-direction. Cut out
the original copy of $W_0$ and replace it by $W$; this can be
done symplectically. The resulting symplectic manifold
$W_1$ with boundary $M_1\cong M$ is simply $W$ with a collar added,
and Theorem~\ref{thm:homology} applies to it.

Inside the boundary $M_0\cong M$ of $W_0$, there is a copy
of $V_0$, which we can think of as $V_0\times\{1\}\subset V_0\times D^2=W_0$.
There is a corresponding copy $V_1$ of $V_0$ inside $M_1$.
Apart from the inclusion $W_0\rightarrow W_1$, there
is a second embedding $W_0\rightarrow W_1$ that maps $V_0$
diffeomorphically to~$V_1$, obtained by extending an isotopy that
moves $V_0$ to~$V_1$.

Now set $X=W_1\setminus\Int(W_0)$, which is a cobordism
between $M_0$ and~$M_1$. This gives us the following diagram of maps.
We shall refer to it in the sequel as the
`cobordism diagram'.

\begin{diagram}
V_0                 &                       &     & \rTo^{\cong}          &                      &                       & V_1\\
                    & \rdTo^{\simeq}        &     &                       &                      & \ldTo^{\mathrm{(ii)}} &    \\
\dTo^{\mathrm{(i)}} &                       & W_0 & \rTo^{\mathrm{(iii)}} & W_1                  &                       & \dTo_{\mathrm{(i')}} \\
                    & \ruTo^{\mathrm{g.p.}} &     &                       & \uTo^{\mathrm{g.p.}} & \luTo^{\mathrm{(o)}}  & \\
M_0                 &            & \rTo^{\mathrm{(iv)}}    &                       & X                    & \lTo^{\mathrm{(ii)}}  & M_1    
\end{diagram}

All maps in this diagram are inclusions, except for $V_0\rightarrow V_1$,
which is the diffeomorphism just mentioned, and $W_0\rightarrow W_1$,
which indicates both the inclusion and the alternative embedding just
described. With this understood, the cobordism diagram is homotopy
commutative.

By Theorem~\ref{thm:homology}, the inclusion $M_1\rightarrow W_1$,
which is essentially the inclusion $M\rightarrow W$, is always
$\pi_1$-surjective. Under the assumption that $M$ is simply
connected, it is of course also $\pi_1$-injective. Since, later on,
we shall be considering more general situations, we formulate the
next result in terms of this assumption.

\begin{lem}
\label{lem:filling-pi}
If the inclusion $M\rightarrow W$ is $\pi_1$-injective, then
the inclusion maps $M_0,M_1\rightarrow X$ are $\pi_1$-isomorphic.
\end{lem}

\begin{proof}
The following steps refer to the maps with the corresponding labels
in the cobordism diagram.

(g.p.) This label stands for `general position'. First consider
the inclusion $M_0\rightarrow W_0$. By assumption, $W_0$ has
a handle decomposition with handles of index at most~$n-1$. This allows us
to define a Morse--Smale function on $W_0$ whose negative gradient
flow contracts $W_0$ onto a subcomplex of dimension at most~$n$,
which we call the skeleton. (Notice that the dimension of the skeleton
may be larger than the homotopical dimension of~$W_0$. We do not care
about homologically inessential handles, as long as they have
subcritical index.) Under the positive gradient flow, the
complement of the skeleton flows into the boundary~$M_0$.

Now consider a relative $k$-disc in $W_0$, that is, a continuous map
$(D^k,S^{k-1})\rightarrow (W_0,M_0)$. By general position,
we can make this disc (rel boundary) disjoint from the skeleton,
provided that $k+n-1<2n$. The gradient flow
of the Morse--Smale function then allows us to push
that disc (rel boundary) into $M_0$.

For $k\in\{1,2\}$,
that inequality is satisfied for all~$n\geq 2$. It follows that
the relative homotopy groups $\pi_1(W_0,M_0)$ and $\pi_2(W_0,M_0)$
are trivial, which implies that the inclusion $M_0\rightarrow W_0$ is
$\pi_1$-isomorphic.

In an analogous fashion, we can deal with the inclusion $X\rightarrow W_1$.
Given a relative $k$-disc $(D^k,S^{k-1})\rightarrow (W_1,X)$,
for $k\in\{1,2\}$ we can again make it disjoint from the skeleton
of $W_0$, and then use the gradient flow to push it into~$X$.

(o) The map $M_1\rightarrow W_1$ is $\pi_1$-isomorphic by
Theorem~\ref{thm:homology} and our assumption.

(i) The inclusion $V_0\rightarrow M_0$ is $\pi_1$-isomorphic, since
both the homotopy equivalence $V_0\rightarrow W_0$ and the
inclusion $M_0\rightarrow W_0$ have this property.

(i') The inclusion $V_1\rightarrow M_1$ is then likewise $\pi_1$-isomorphic.

(ii) With (i') it follows that $V_1\rightarrow W_1$ is
$\pi_1$-isomorphic, and the same is true for $M_1\rightarrow X$
by (g.p.). 

(iii) Interpreting $W_0\rightarrow W_1$ as the `alternative embedding',
we conclude from (ii) that this map is $\pi_1$-isomorphic.

(iv) It now follows that $M_0\rightarrow X$ is $\pi_1$-isomorphic.
\end{proof}

Next we analyse the maps in the cobordism diagram with a view to homology.
The following lemma does not presuppose any information on
the fundamental group of~$M$. The input previously provided
by a general position argument or by the assumption on $\pi_1$-injectivity
now comes directly from Theorem~\ref{thm:homology}.

\begin{lem}
\label{lem:filling-H}
The relative homology groups $H_k(X,M_0)$ and $H_k(X,M_1)$
vanish for all $k\in\N_0$.
\end{lem}

\begin{proof}
By Theorem~\ref{thm:homology}, the inclusions $M_0\rightarrow W_0$
and $M_1\rightarrow W_1$ induce isomorphisms on $H_k$ for
$k=0,\ldots,\ell_0$.

(i), (i') It follows that the inclusion $V_0\rightarrow M_0$
induces isomorphisms in homology up to degree~$\ell_0$, and
hence so does the inclusion $V_1\rightarrow M_1$.

(ii) The same is then true for the inclusion $V_1\rightarrow W_1$.

(iii) Since the homology groups of $W_0$ and $W_1$
in degree $k>\ell_0$ are trivial by Theorem~\ref{thm:homology}, we conclude
that the inclusion $W_0\rightarrow W_1$ induces an isomorphism
in homology, and hence $H_k(W_1,W_0)=0$ for all~$k$.

(iv) By excision we have $H_k(X,M_0)\cong H_k(W_1,W_0)=0$.
With Poincar\'e duality and the universal coefficient theorem we
conclude $H_k(X,M_1)=0$.
\end{proof}

\begin{proof}[Proof of Theorem~\ref{thm:filling}]
By Lemmata \ref{lem:filling-pi} and~\ref{lem:filling-H}
and the relative Hurewicz theorem, the simply connected
cobordism $\{M_0,X,M_1\}$ is an $h$-cobordism. Hence, as $n\geq 3$,
it is diffeomorphic to a product $[0,1]\times M$ by the
$h$-cobordism theorem. It follows that $W$, which is
diffeomorphic to~$W_1$, is obtained from
$W_0$ by attaching this collar $[0,1]\times M$, so $W$ and $W_0$
are diffeomorphic.
\end{proof}

A closer inspection of the argument in this section immediately
yields the following generalisation of Theorem~\ref{thm:filling}.
For a classical survey on Whitehead groups and Whitehead torsion
see~\cite{miln66}.

\begin{thm}
\label{thm:filling2}
Let $(M,\xi)$ be a manifold as in Theorem~\ref{thm:homology}
with $\pi_1(M)$ finite. Then the symplectically aspherical
fillings $(W,\omega)$ of $(M,\xi)$ for which the inclusion
$M\rightarrow W$ is $\pi_1$-injective (which by Theorem~\ref{thm:homology}
are all symplectically aspherical fillings when $\pi_1(M)$ is
finite abelian) are pairwise homotopy equivalent. If, in addition, $n\geq 3$
and the fundamental group $\pi_1(M)$ has trivial Whitehead group
$\Wh(\pi_1(M))=0$, then these fillings are pairwise diffeomorphic.
\end{thm}

\begin{proof}
As observed in the proof of Lemma~\ref{lem:filling-pi},
the inclusion $M\rightarrow W_0$ into the subcritical
Stein filling is always $\pi_1$-isomorphic by a general
position argument. The same is true for the inclusion $M\rightarrow W$
by assumption and Lemma~\ref{lem:filling-pi}. It follows
that the (compact!)\ universal cover of $(M,\xi)$ is filled
by either of the universal covers of $(W_0,\omega_0)$ and $(W,\omega)$.

As before, we now investigate the cobordism $\{M_0,X,M_1\}$.
Again by Lemma~\ref{lem:filling-pi}, here too we can pass
to the compact cobordism
$\{\wtM_0,\wtX,\wtM_1\}$ of universal covers.
Our previous argument shows that this is an $h$-cobordism,
hence so is $\{M_0,X,M_1\}$. It follows that $W_1$, which is
a diffeomorphic copy of~$W$, is homotopy equivalent to~$W_0$,
since $W_1=W_0\cup_{M_0}\{M_0,X,M_1\}$.

Under the assumption $\Wh(\pi_1(M))=0$, the cobordism
$\{M_0,X,M_1\}$ is an $s$-cobordism,
and hence trivial for $n\geq 3$.
\end{proof}
\section{Coverings}
\label{section:coverings}
Starting from a symplectically aspherical filling
$(W,\omega)$ of $(M,\xi)$ as in Theorem~\ref{thm:homology},
we now analyse the argument for a
covering $W'\rightarrow W$.
This covering is not assumed to be finite, so it includes the case
of the universal covering $\wtW\rightarrow W$ when
$\pi_1(W)$ has infinite order. Our main applications will
concern the situation when the universal cover $\wtW_0$
of the presumed subcritical Stein filling $(W_0,\omega_0)$
of $(M,\xi)$ is contractible.

We shall assume throughout that $W'$ is connected.
The covering $W'\rightarrow W$ induces a covering
$\partial W'=:M'\rightarrow M$.
Since, by Theorem~\ref{thm:homology}~(b), the inclusion $M\rightarrow W$
is $\pi_1$-surjective, the manifold $M'$ must likewise be connected,
as is seen by a standard covering space argument.

As observed in the proof of Lemma~\ref{lem:filling-pi}, the
inclusion $M\rightarrow W_0$ is $\pi_1$-isomorphic thanks to $W_0$
being contractible onto its skeleton of dimension at most~$n-1$.
So there is a covering $W_0'\rightarrow W_0$ inducing the
covering $M'\rightarrow M$ on the boundary, and a corresponding
covering $V'\rightarrow V$ in the notation of
Section~\ref{subsection:completion}.
Given this information, we can then define symplectic
manifolds $(Z',\Omega')$ and $(\hat{Z}',\hat{\Omega}')$
in complete analogy with the construction in
that section. Write $J'$ for the lifted almost complex structure
on $\hat{Z}'$. Then $J'$ is uniformly tamed by $\hat{\Omega}'$,
and the metric $g':=\hat{\Omega}'(\,.\,,J'\,.\,)$ is complete and admits
both an upper bound on the sectional curvature and a positive lower
bound on the injectivity radius, that is,
$(\hat{Z}',\hat{\Omega}')$ is geometrically bounded in the
sense of \cite[Definition~2.2.1]{alp94}.

As before, we define a moduli space $\MM'$ of holomorphic
spheres $u'\co\CP^1\rightarrow (\hat{Z}',J')$ subject to
the analogous conditions (M1) and (M2). The composition of
such holomorphic spheres with the covering map
$\mfp\co \hat{Z}'\rightarrow\hat{Z}$
defines a covering $\MM'\rightarrow\MM$.

Proposition~\ref{prop:ev} holds unchanged; we need only establish
properness of the evaluation map in the new setting.

\begin{lem}
The evaluation map
\[ \begin{array}{rccc}
\ev'\co & \MM'\times\CP^1 & \longrightarrow & \hat{Z}'\\
        & (u',z)          & \longmapsto     & u'(z)
\end{array}\]
is proper.
\end{lem}

\begin{proof}
Given a compact subset $K\subset \hat{Z}'$, consider
a sequence $(u_{\nu}',z_{\nu})$ in $(\ev')^{-1}(K)$.
We may assume that $z_{\nu}\rightarrow z_0\in\CP^1$ and
$u_{\nu}'(z_{\nu})\rightarrow p_0'\in K$ for $\nu\rightarrow\infty$.
By Proposition~\ref{prop:ev}, we may further assume that
the sequence $u_{\nu}:=\mfp\circ u_{\nu}'$ of holomorphic
curves in $\hat{Z}$ is $C^{\infty}$-convergent.

For a given $w\in\CP^1$, let $\gamma_{\nu}$ be a unit speed
geodesic in $\CP^1$ of length $L_{\nu}\leq\pi/2$ with
respect to the Fubini--Study metric, connecting $w$ with~$z_{\nu}$.
Then, thanks to $(\hat{Z}',\hat{\Omega}')$
being geometrically bounded,
the distance between $u_{\nu}'(w)$ and $u_{\nu}'(z_{\nu})$
in $\hat{Z}'$ with respect to the metric $g'$
can be estimated from above by
\[ \mathrm{dist}\bigl((u_{\nu}'(w),u_{\nu}'(z_{\nu})\bigr)\leq
\int_0^{L_{\nu}}\Bigl|\frac{\rmd}{\rmd t}(u_{\nu}'\circ\gamma_{\nu})\Bigr|
\,\rmd t\leq
\mathrm{const.} \cdot \|Tu_{\nu}\|_{C^0}\leq\mathrm{const.},\]
with constants that do not depend on~$w$.
Hence, the images $u_{\nu}'(\CP^1)$ are all contained
in a sufficiently large closed metric ball about $p_0'$,
which is a compact subset of the complete Riemannian manifold
$(\hat{Z}',g')$.

This guarantees the existence of a Gromov-convergent subsequence
of $(u_{\nu}')$.
The argument then concludes as in the proof of Proposition~\ref{prop:ev}.
\end{proof}

The arguments in Sections~\ref{subsection:hom-epi}
and~\ref{subsection:proof-homology-a} then go
through, \emph{mutatis mutandis}, for the covering spaces, with the proviso
that cohomology be replaced by cohomology with compact supports
in all arguments requiring (implicitly or explicitly) Poincar\'e
duality. See \cite[pp.~242--249]{hatc02} for a good exposition
of Poincar\'e duality in this context.

In particular, again we obtain a homology epimorphism
$H_k(V')\rightarrow H_k(W')$ for all~$k$. The following proposition
is the simplest consequence of this fact, but one that has
wide-ranging applications, as we shall see.

\begin{prop}
\label{prop:contractible}
With $M,W_0,W$ as in Theorem~\ref{thm:homology},
suppose that the inclusion $M\rightarrow W$ is $\pi_1$-injective
and $\wtW_0$ is contractible. Then $\wtW$ is likewise contractible.
\end{prop}

\begin{proof}
Under the assumption that the inclusion $M\rightarrow W$ is
$\pi_1$-injective and hence, by Theorem~\ref{thm:homology}~(b),
$\pi_1$-isomorphic, the universal covering $\wtW\rightarrow W$
restricts to the universal covering $\wtM\rightarrow M$
on the boundary. The assumption on $\wtW_0$ being
contractible is the same as saying that the universal cover
$\widetilde{V}$ is contractible. The mentioned homology epimorphism
then implies that $\wtW$ is a simply connected space with vanishing
reduced homology, and hence contractible.
\end{proof}

In Section~\ref{section:simple} we shall make use of the following lemma.

\begin{lem}
\label{lem:M0-sur-X}
With $M,W_0,W$ as in Theorem~\ref{thm:homology},
suppose that the inclusion $M\rightarrow W$ is $\pi_1$-injective.
Then the inclusion $\wtM_0\rightarrow\wtX$ induces a surjective
homomorphism in homology.
\end{lem}

\begin{proof}
The assumption on $\pi_1$ allows us to pass to universal
covers in the cobordism diagram in Section~\ref{section:filling}.
Consider the following commutative diagram with exact rows:
\begin{diagram}
H_k(\wtM_0) & \rTo       & H_k(\wtX)   & \rTo^i     & H_k(\wtX,\wtM_0)\\
\dTo        &            & \dTo        &            & \dTo_{\cong}\\
H_k(\wtW_0) & \rTo^{j_1} & H_k(\wtW_1) & \rTo^{j_2} & H_k(\wtW_1,\wtW_0)
\end{diagram}
The vertical map on the right is the excision isomorphism. The homology
epimorphism $H_k(V')\rightarrow H_k(W')$ for any covering gives us, in
particular, an epimorphism $H_k(\widetilde{V}_1)\rightarrow H_k(\wtW_1)$.
From the cobordism diagram we then see
that $j_1$ is likewise surjective, and hence $j_2$ the zero
homomorphism. This in turn implies that $i$ is the zero
homomorphism.
\end{proof}
\section{Handle decompositions}
\label{section:handle}
In this section we discuss the homotopy and diffeomorphism
classification of fillings of a given closed, connected
contact manifold $(M,\xi)$  of dimension $2n-1$ under the assumption that
information is given on the maximal index in a handle
decomposition of the filling. Applications include the homotopy
classification of subcritical Stein fillings.

Thus, let $M,W,W_0$ be as in Theorem~\ref{thm:homology},
with the additional assumption that $W$
has a handle decomposition involving handles of
index $\leq\ell$ only. Then a general position argument as in the
proof of Lemma~\ref{lem:filling-pi} yields the following
result.

\begin{lem}
\label{lem:rel-pi-H}
For $k\leq 2n-1-\ell_0$, the relative groups $\pi_k(W_0,M)$ and
$H_k(W_0,M)$ are trivial.
The same is true for the relative groups $\pi_k(W_1,X)$ and $H_k(W_1,X)$.
For $k\leq 2n-1-\ell$, the relative groups $\pi_k(W,M)$ and
$H_k(W,M)$ are trivial.
\qed
\end{lem}

In particular, for $\ell\leq 2n-3$ the inclusion $M\rightarrow W$
is $\pi_1$-isomorphic. Then, by the proof of Lemma~\ref{lem:filling-pi},
the same will be true for all the other maps in the cobordism diagram,
so that we can pass simultaneously to
the universal covers of all spaces in that diagram. From now on, this
assumption will be understood.

\begin{thm}
\label{thm:W-simeq-W0}
If $\ell_0+\max(\ell_0,\ell)\leq 2n-2$, then $W$ and $W_0$ are
homotopy equivalent.
\end{thm}

\begin{proof}
From the lemma it follows that all maps in the cobordism diagram,
also at the level of universal covers,
are $\pi_k$- and $H_k$-isomorphic for $k\leq 2n-2-\max(\ell_0,\ell)$.
By the assumption in the theorem, this holds in particular
for $k\leq\ell_0$. 

This implies that the inclusion $\wtW_0\rightarrow\wtW_1$
induces an isomorphism in homology in all degrees, since
\[ H_k(\wtW_0)=0=H_k(\wtW_1)\;\;\text{for}\;\; k\geq\ell_0+1;\]
for $\wtW_0$ this follows from the homotopical assumptions;
for $\wtW_1\cong\wtW$, from the homology epimorphism
in Section~\ref{section:coverings}.

The homology exact sequence
of the pair $(\wtW_1,\wtW_0)$ then shows the vanishing
of $H_k(\wtW_1,\wtW_0)$ in all degrees. By excision,
we also have $H_k(\wtX,\wtM_0)=0$ for all~$k$.

With the relative Hurewicz theorem we deduce $\pi_k(\wtX,\wtM_0)=0$
for all~$k$. Since the inclusion $M_0\rightarrow X$ is already
known to be a $\pi_1$-isomorphism, $M_0$ is a strong deformation retract
of $X$ by Whitehead's theorem \cite[Theorem~4.5]{hatc02}. Hence
\[ W\simeq W_1=W_0\cup_{M_0}X\simeq W_0,\]
as we wanted to show.
\end{proof}

The proof of the result that all subcritical Stein fillings
of a given contact manifold are homotopy equivalent is
now straightforward.

\begin{proof}[Proof of Theorem~\ref{thm:scS-homotopy}]
For a subcritical filling we have $\ell\leq n-1$, and hence
$\ell_0+\max(\ell_0,\ell)\leq 2n-2$.
\end{proof}

The required estimate is also satisfied for $\ell\leq n$,
provided that $\ell_0\leq n-2$. This gives
the following corollary.

\begin{cor}
If $(M,\xi)$ admits a subcritical Stein filling with $\ell_0\leq n-2$,
then all Stein fillings (including critical ones) are homotopy equivalent.
\end{cor}
\section{Simple spaces}
\label{section:simple}
In this section we prove Theorem~\ref{thm:simple}.
Thus, consider manifolds $M,W_0,W$ as in Theorem~\ref{thm:homology}
under the additional assumption that $M$ is a simple space
and $n\geq 3$. In particular, since the
action of $\pi_1$ on itself is given by conjugation, $\pi_1(M)$
must be abelian. Thus, by Theorem~\ref{thm:homology}~(b) and
Lemma~\ref{lem:filling-pi}, all maps in the cobordism diagram
are $\pi_1$-isomorphic, and we
can pass to universal covers. Also, from Lemma~\ref{lem:filling-H}
we know that the relative homology groups $H_k(X,M_0)$ and
$H_k(X,M_1)$ vanish for all~$k$.

The assumptions of Theorem~\ref{thm:simple} are taken for granted
in this section.

\begin{lem}
The relative homotopy groups $\pi_k(X,M_0)$ are trivial.
\end{lem}

\begin{proof}
The statement holds for $k=0,1$, since $M_0$ and $X$ are connected, and the
inclusion map $M_0\rightarrow X$ is $\pi_1$-isomorphic.

Inductively, we assume that the vanishing of $\pi_i(X,M_0)$ has been
established for $i\leq k-1$. We want to show $\pi_k(X,M_0)=0$.

Write $\gamma(\eta)\in\pi_k(X,M_0)$ for the element obtained by
the action of $\gamma\in\pi_1(M_0)$ on $\eta\in\pi_k(X,M_0)$.
By the relative Hurewicz theorem, the Hurewicz homomorphism
\[ h_k\co \pi_k(X,M_0)\longrightarrow H_k(X,M_0)\]
is an epimorphism whose kernel is the subgroup
of $\pi_k(X,M_0)$ generated by elements of the form $\gamma(\eta)-\eta$.
Notice that the inclusion $M_0\rightarrow X$ being
$\pi_1$-isomorphic implies that $\pi_2(X,M_0)$ is isomorphic to a quotient
group of $\pi_2(X)$, and hence abelian; so are all higher relative
homology groups. Since $H_k(X,M_0)=0$, the kernel of $h_k$ is the full group.

The action of $\pi_1$, by its definition, commutes with the
boundary homomorphism $\partial\co\pi_k(X,M_0)\rightarrow\pi_{k-1}(M_0)$.
Hence
\[ \partial\bigl(\gamma(\eta)-\eta\bigr)=\gamma(\partial\eta)-
\partial\eta=0,\]
as $M_0$ is a simple space. Thus, $\partial$ is the zero homomorphism.

Consider the commutative diagram
\begin{diagram}
\pi_k(\wtX,\wtM_0) & \rTo^{\tilde{\partial}} & \pi_{k-1}(\wtM_0)\\
\dTo^{\cong}       &                         & \dTo\\
\pi_k(X,M_0)       & \rTo^{\partial=0}       & \pi_{k-1}(M_0).
\end{diagram}
The vertical homomorphism on the left is an isomorphism for
all $k\geq 2$. For $k\geq 3$, this is a general consequence of the
five-lemma; for $k=2$ one needs to use that $\partial=0$.
The vertical homomorphism on the right is an isomorphism
for $k\geq 3$; for $k=2$ we have $\pi_{k-1}(\wtM_0)=0$.
In either case we conclude that $\tilde{\partial}$ is also the
zero homomorphism.

Next we consider the commutative diagram coming
from the `homotopy-homo\-logy ladder' of the pair $(\wtX,\wtM_0)$:
\begin{diagram}
          &        & \pi_k(\wtX,\wtM_0) &
   \rTo^{\tilde{\partial}=0} & \pi_{k-1}(\wtM_0)\\
          &        & \dTo^{\cong}       &
                             & \dTo \\
H_k(\wtX) & \rTo^i & H_{k}(\wtX,\wtM_0) &
   \rTo^j                    & H_{k-1}(\wtM_0)
\end{diagram}
The vertical isomorphism on the left is the Hurewicz isomorphism.
The commutative square implies the triviality of the homomorphism~$j$.
The homomorphism $i$ was shown to be trivial in the proof
of Lemma~\ref{lem:M0-sur-X}.
We conclude $H_k(\wtX,\wtM_0)=0$, hence $\pi_k(X,M_0)=\pi_k(\wtX,\wtM_0)=0$.
\end{proof}

\begin{lem}
The relative homotopy groups $\pi_k(X,M_1)$ are trivial.
\end{lem}

\begin{proof}
Again we argue inductively; the inductive assumption
for $k=0,1$ is satisfied. Assume that $\pi_i(X,M_1)$ vanishes for
$i\leq k-1$. From the preceding lemma we know that $M_0$
is a deformation retract of~$X$. It follows that $X$ is likewise
a simple space. As in the foregoing proof we see that $\pi_k(X,M_1)$
is generated by elements of the form $\gamma(\zeta)-\zeta$ with
$\zeta\in\pi_k(X,M_1)$ and $\gamma\in\pi_1(M_1)$. Again,
$M_1$ being simple implies the triviality of the boundary
homomorphism $\partial\co\pi_k(X,M_1)\rightarrow
\pi_{k-1}(M_1)$. Thus, the homomorphism $\pi_k(X)\rightarrow
\pi_k(X,M_1)$ is surjective. It follows that $(X,M_1)$ is simple.
With the information on the generators of $\pi_k(X,M_1)$, this
shows that $\pi_k(X,M_1)$ is trivial.
\end{proof}

\begin{proof}[Proof of Theorem~\ref{thm:simple}]
The two lemmata show that $\{M_0,X,M_1\}$ is an $h$-cobordism,
and hence an $s$-cobordism under the assumption $\Wh(\pi_1(M))=0$.
\end{proof}

For an application of this theorem see Example~\ref{ex:Q1}.
\section{Unit stabilised cotangent bundles}
\label{section:cotangent}
We now return to contact manifolds of the kind described in
Example~\ref{ex:distinguish}. Given a closed Riemannian manifold~$Q$
of dimension~$q$,
consider the unit sphere bundle $M:=S(T^*Q\oplus\C^m)$, $m\geq 1$,
of the $m$-fold stabilised cotangent bundle of~$Q$. Then
$\dim M=2n-1$ with $n=q+m$. We always
equip this manifold with the canonical contact structure $\xi$
given by the contact form
\[ \lambda_Q+\frac{1}{2}\sum_{j=1}^m(x_j\,\rmd y_j-y_j\,\rmd x_j).\]

\begin{thm}
\label{thm:cotangent1}
Let $M=S(T^*Q\oplus\C^m)$, $m\geq 1$, with its standard contact
structure~$\xi$, where $Q$ is a closed $q$-dimensional manifold
subject to the following conditions:
\begin{itemize}
\item[(i)] $Q$ is aspherical, i.e.\ the universal cover $\widetilde{Q}$
is contractible.
\item[(ii)] The fundamental group $\pi_1(Q)$ is abelian and it
has trivial Whitehead group $\Wh(\pi_1(Q))$.
\item[(iii)] $n=q+m\geq 3$, that is, $\dim M\geq 5$.
\end{itemize}
Then every symplectically aspherical filling of $(M,\xi)$
is diffeomorphic to the total space of the disc bundle
$W_0:=D(T^*Q\oplus\C^m)$.
\end{thm}

\begin{proof}
The assumption (ii) on $\pi_1(Q)$ being abelian implies that the
inclusion $M\rightarrow W$, for any symplectically aspherical
filling $(W,\omega)$ of $(M,\xi)$, is $\pi_1$-isomorphic by
Theorem~\ref{thm:homology}~(b), and so all maps in the cobordism diagram
are $\pi_1$-isomorphic
by Lemma~\ref{lem:filling-pi} and its proof. This allows us to pass to
universal covers in that diagram.

\begin{rem}
If the assumption on $\pi_1(Q)$ being abelian is dropped, the
conclusion of the theorem still holds true for all fillings $(W,\omega)$
for which the inclusion $M\rightarrow W$ is $\pi_1$-injective.
\end{rem}

Since $\widetilde{Q}$ is contractible ($\widetilde{Q}\simeq*$) by
assumption~(i), we also have $\wtW_0\simeq*$.
Proposition~\ref{prop:contractible} then tells us that $\wtW\simeq*$.
From the homology sequence of the pair $(\wtW_1,\wtW_0)$,
where $W_1$ is the diffeomorphic copy of $W$ in the notation of the
preceding sections, we see that $H_k(\wtW_1,\wtW_0)=0$ in
all degrees, and $H_k(\wtX,\wtM_0)=0$ by excision.

Thus, as in the proof of Theorem~\ref{thm:W-simeq-W0} we find that
$M_0$ is a strong deformation retract of~$X$. In particular, we
have $H_k(X,M_0)=0$ for all~$k$, and hence $H_k(X,M_1)=0$ for
all~$k$ by Poincar\'e duality and the universal coefficient theorem.

In order to show that the `upper' inclusion $M_1\rightarrow X$
is likewise a homotopy equivalence, we need to establish
$\pi_k(X,M_1)=0$ for all~$k$; we already know this for $k=0,1$.
To this end, analogous to the proof of Theorem~\ref{thm:W-simeq-W0}, we have
to show the vanishing of the relative homology groups $H_k(\wtX,\wtM_1)$.

From $H_k(\wtX,\wtM_0)=0$ for all $k$ we know that
$H_k(\wtX)$ is isomorphic to $H_k(\wtM)$ for all~$k$.
Hence, if $H_k(\wtM)=0$, then $H_k(\wtX)=0$, and the inclusion
$\wtM_1\rightarrow\wtX$ obviously induces an isomorphism on~$H_k$;
the same is true on~$H_0$.
In our situation, the only non-zero homology group of $\wtM$
in higher degree is $H_{q+2m-1}(\wtM)\cong\Z$.

From the Gysin homology sequence of the sphere bundle $M\rightarrow Q$
we see that $H_{q+2m-1}(M)\cong\Z$, generated by the fibre
class. The same is true for the universal cover $\wtM$, which implies
that the homomorphism
\[ \Z\cong H_{q+2m-1}(\wtM)\longrightarrow H_{q+2m-1}(M)\cong\Z\]
is an isomorphism.
The relevant part of the homology exact sequences of the pairs
$(\wtX,\wtM_0)$ and $(X,M_0)$ becomes
\begin{diagram}
H_{q+2m-1}(\wtM_0) & \rTo^{\cong} & H_{q+2m-1}(\wtX)\\
\dTo^{\cong}       &              &\dTo             \\
H_{q+2m-1}(M_0)    & \rTo^{\cong} & H_{q+2m-1}(X),
\end{diagram}
so the homomorphism
\[ \Z\cong H_{q+2m-1}(\wtX)\longrightarrow H_{q+2m-1}(X)\cong\Z\]
is likewise an isomorphism.

Finally, from the homology exact sequences of the pairs
$(\wtX,\wtM_1)$ and $(X,M_1)$ we have
\begin{diagram}
H_{q+2m-1}(\wtM_1) & \rTo         & H_{q+2m-1}(\wtX)\\
\dTo^{\cong}       &              &\dTo_{\cong}      \\
H_{q+2m-1}(M_1)    & \rTo^{\cong} & H_{q+2m-1}(X),
\end{diagram}
which gives us an isomorphism
\[ \Z\cong H_{q+2m-1}(\wtM_1)\longrightarrow H_{q+2m-1}(\wtX)\cong\Z.\]
Thus, the inclusion $\wtM_1\rightarrow\wtX$ is a homology isomorphism.

It follows that $\{ M_0,X,M_1\}$ is an $h$-cobordism. Under the
assumption that $\Wh(\pi_1(Q))=0$ it is an $s$-cobordism, and hence trivial
under the dimension assumption~(iii).
\end{proof}

\begin{ex}
\label{ex:Q1}
(1) Any closed Riemannian manifold $Q$ with abelian fundamental group
and non-positive sectional curvature satisfies the assumptions
of the theorem, since $Q$ is aspherical by the Hadamard--Cartan
theorem, and $\Wh(\pi_1(Q))$ is trivial by the work of Farrell and
Jones~\cite{fajo90}.

(2) The conclusions of the theorem hold whenever $Q$
is a product of unitary groups and spheres. Indeed,
in this case $\pi_1(M)\cong\pi_1(Q)$ is a free abelian group,
which has trivial Whitehead group. The manifold $Q$ has
a trivial stable tangent bundle, and $Q$ is a simple
space, hence so is $M\cong Q\times S^{q+2m-1}$. Then appeal
to Theorem~\ref{thm:simple}.
\end{ex}

When information is given on the handle structure of the
filling, as in the next theorem, we can use the
results from Section~\ref{section:handle} to remove the
condition on $\pi_1(Q)$ being abelian.

\begin{thm}
\label{thm:cotangent2}
Let $M=S(T^*Q\oplus\C^m)$, $m\geq 1$, with its standard contact
structure~$\xi$, where $Q$ is a closed $q$-dimensional manifold
subject to the following conditions:
\begin{itemize}
\item[(i)] $Q$ is aspherical.
\item[(ii)] $\Wh(\pi_1(Q))$ is trivial.
\item[(iii)] $n=q+m\geq 3$, that is, $\dim M\geq 5$.
\end{itemize}
Then the following holds:
\begin{itemize}
\item[(a)] All subcritical Stein fillings of $M$ are diffeomorphic to
$W_0:=D(T^*Q\oplus\C^m)$.
\item[(b)] If $m\geq 2$, then all Stein fillings of $M$ are
diffeomorphic to~$W_0$.
\end{itemize}
\end{thm}

\begin{proof}
In the notation of Lemma~\ref{lem:rel-pi-H} we have
$\ell_0=q\leq n-1$ and $\ell\leq n-1$ in case~(a); $\ell_0=q\leq n-2$
and $\ell\leq n$ in case~(b). In either case,
the inclusion $M\rightarrow W$ (where $W$ is any filling of the
described type) is $\pi_1$-isomorphic by Lemma~\ref{lem:rel-pi-H}.
 
Also, the assumption of Theorem~\ref{thm:W-simeq-W0} is satisfied,
so the argument there shows that the `lower' inclusion
$M_0\rightarrow X$ in the cobordism $\{ M_0,X,M_1\}$
is a homotopy equivalence.

The argument for the `upper' inclusion $M_1\rightarrow X$
is then as in the preceding proof.
\end{proof}

\begin{ex}
\label{ex:Q2}
The conclusions of Theorem~\ref{thm:cotangent2} hold for the following
manifolds~$Q$.

(1) Any closed surface. For orientable surfaces of genus at least $1$
and non-orientable surfaces of genus at least~$2$,
Theorem~\ref{thm:cotangent2} applies directly thanks to the
results cited in Example~\ref{ex:Q1}. For $Q=S^2$ even
Theorem~\ref{thm:cotangent1} holds true; this example
is covered by Theorem~\ref{thm:filling}.

For $Q=\RP^2$, the Whitehead group of the fundamental group
$\pi_1(\RP^2)=\Z_2$ vanishes. We claim that, again,
Theorem~\ref{thm:cotangent1} holds true in this case. Indeed,
the only part of the argument that needs to be adapted is
where we show that $H_k(\wtX,\wtM_1)$ vanishes for all~$k$.
As long as we only pass to finite covers $X',M_1'$, this vanishing result
holds by Poincar\'e duality. Thus, for the argument in the
proof of Theorem~\ref{thm:cotangent1} to go through it suffices
to find a finite cover $M_1'$ of $M_1$ such that the projection map
$\wtM_1\rightarrow M_1'$ induces an isomorphism on $H_k$
whenever $H_k(\wtM_1)$ is non-trivial. In the present example,
we can take $M_1'=\wtM_1=S^2\times S^{2m+1}$. Of course, in this
case of a finite fundamental group we can alternatively appeal
directly to Theorem~\ref{thm:filling2}.

(2) Any closed, irreducible aspherical $3$-manifold.
This follows from the result of Roushon~\cite{rous11} that the
fundamental group of such a manifold has a trivial Whitehead group.

(3) Any closed, irreducible $3$-manifold $Q$ covered
by $S^3$ with $\Wh(\pi_1(Q))=0$. Here the argument is as in (1).
Examples of the allowed fundamental groups are $\Z_2$, $\Z_3$,
$\Z_4$, $\Z_6$.
\end{ex}
\section{Symplectomorphism type}
\label{section:symplecto}
We now assume that $(M,\xi)$ is a closed, connected contact manifold
of dimension $2n-1$, $n\geq 3$, that admits a $2$-subcritical Stein filling
$(W_0,\omega_0)$,
that is, where the Stein handles are all of index at most $n-2$.
In particular, we have $\ell_0\leq n-2$. Under this assumption, we want
to formulate topological conditions on $M$ that allow us to
classify all subcritical fillings up to symplectomorphism.

For the notion of Stein deformation equivalence
see~\cite[p.~311]{ciel12}. Deformation equivalent
Stein fillings are symplectomorphic in the
sense of~\cite[p.~318]{ciel12}. The concept of
flexible Stein fillings is defined in \cite[Definition~11.29]{ciel12}.
Suffice it to say here that all subcritical Stein fillings
are flexible.

\begin{thm}
\label{thm:flexible}
Let $(M,\xi)$ be a $(2n-1)$-dimensional closed, connected contact manifold,
$n\geq 3$, admitting a $2$-subcritical Stein filling $(W_0,\omega_0)$.
Further, we make the following topological assumptions:
\begin{itemize}
\item[(i)] $\Wh(\pi_1(M))=0$.
\item[(ii)] $M$ (or some finite cover of~$M$) is
a simple space, or $M$ (or some finite cover) has the property
that the homomorphism $H_k(\wtM)\rightarrow H_k(M)$ is an isomorphism
whenever $H_k(\wtM)\neq 0$.
\end{itemize}
Then all flexible Stein fillings of $(M,\xi)$ are Stein deformation
equivalent.
\end{thm}

\begin{ex}
The assumptions of the theorem are satisfied by $(S^{2n-1},\xist)$
and by any sphere bundle $S(T^*Q\oplus\C^m)$ with its standard
contact structure, provided $m\geq 2$, $\Wh(\pi_1(Q))=0$,
and $Q$ is a Lie group or satisfies the homological assumption (ii)
on coverings.
\end{ex}

\begin{proof}[Proof of Theorem~\ref{thm:flexible}]
Given a flexible Stein filling $(W,\omega)$ of~$(M,\xi)$,
we consider the cobordism $\{M_0,X,M_1\}$ as in Section~\ref{section:filling}.
As we saw  in the proof of Theorem~\ref{thm:W-simeq-W0}, the
lower inclusion $M_0\rightarrow X$ is a homotopy equivalence.
The topological condition (ii) guarantees that the upper
inclusion is likewise a homotopy equivalence, see
Section~\ref{section:simple} for the case that $M$ is a simple space,
and the argument in Example~\ref{ex:Q2}~(2) for the case
when the homological information on $\wtM$ is given.
Together with condition~(i) this implies that $\{M_0,X,M_1\}$
is an $s$-cobordism, so that $W$ and $W_0$ are diffeomorphic.

We now want to construct a Stein structure on the
cobordism $\{M_0,X,M_1\}$. To this end,
consider a Stein structure on $\C$ with two $0$-handles
and one $1$-handle; the gradient flow of the plurisubharmonic potential
is shown in Figure~\ref{figure:stein-C}.
The assumption on $W_0$ being $2$-subcritical translates into saying
that the Stein manifold $(V,J_V)$ --- in the notation of
Section~\ref{section:proof-homology} --- has a plurisubharmonic
potential $\psi_V$ with finitely many critical points up to index $n-2$ only.
On the product $V\times\C$ we then still have a subcritical potential~$\psi$.
Each critical point of $\psi_V$ gives rise to three critical points
of~$\psi$; the ones sitting over the index~$1$ point in $\C$ have
their index shifted up by~$1$.

\begin{figure}[h]
\labellist
\small\hair 2pt
\endlabellist
\centering
\includegraphics[scale=0.4]{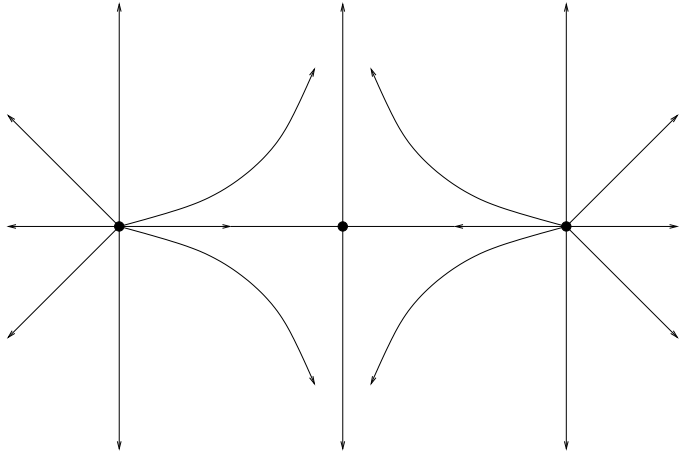}
  \caption{A Stein structure on $\C$.}
  \label{figure:stein-C}
\end{figure}

Now we build the cobordism $X$ as before, where the copy
of $W_0$ that is replaced by $W$ and the copy of $W_0$ that
is removed are each placed, in the $\C$-direction, in a neighbourhood
of one of the critical points of index~$0$. This induces a
Stein structure on the cobordism. This structure is flexible
since, apart from the flexible handles coming from $W$, it only
contains subcritical handles.
So the theorem follows from the Stein $h$-cobordism
theorem~\cite[Corollary~15.12]{ciel12}.
\end{proof}
\begin{ack}
We thank Paul Biran and Dietmar Salamon for useful conversations.
The commutative diagrams have been produced with Paul Taylor's
\TeX\ macros.
\end{ack}

\end{document}